\documentclass{amsart}
\usepackage{amsthm,amsfonts,amsmath,amssymb,latexsym,epsfig,upref,eucal,ae}
\usepackage{stmaryrd}
\usepackage{mathrsfs}  
\usepackage{tikz-cd}

\usepackage{cite}
\usepackage{scalerel,stackengine}
\usepackage{textcomp}
 
\usepackage[all]{xy}
\usepackage{color}

\theoremstyle{plain}
\newtheorem{theorem}{Theorem}[section]
\newtheorem{lemma}[theorem]{Lemma}
\newtheorem{corollary}[theorem]{Corollary}
\newtheorem{proposition}[theorem]{Proposition}

\theoremstyle{definition}
\newtheorem{definition}[theorem]{Definition}
\newtheorem{definition-theorem}[theorem]{Definition-Theorem}
\newtheorem{example}[theorem]{Example}
\theoremstyle{remark}
\newtheorem{remark}[theorem]{Remark}

\def\Aut{\mathrm{Aut}}

\def\bfi{\mathbf{i}}
\def\bfD{\mathbf{D}}

\def\cF{\mathcal{F}}
\def\cE{\mathcal{E}}
\def\cO{\mathcal{O}}

\def\cT{\mathcal{T}}

\def\cB{\mathcal{B}}

\def\fK{\mathfrak{K}}

\def\K{\mathbb{K}}
\def\R{\mathbb{R}}
\def\Q{\mathbb{Q}}
\def\Z{\mathbb{Z}}
\def\N{\mathbb{N}}

\def\P{\mathbb{P}}
\def\F{\mathbb{F}}

\def\C{\mathbb{C}}

\def\E{\mathbb{E}}

\def\>{\rangle}

\def\<{\langle}
\def\>{\rangle}

\def\Hom{\mathrm{Hom}}
\def\Spec{\mathrm{Spec}}
\def\rk{\mathrm{rk}}
\def\deg{\mathrm{deg}}

\def\Ref{\mathfrak{Ref}}                     
\def\Div{\mathrm{Div}}           
\def\WDiv{\mathrm{WDiv}}                    

\def\Filt{\mathfrak{Filt}}
\def\fP{\mathfrak{P}}

\begin{document}

\title[Stable sheaves and toric GIT]
{Equivariant stable sheaves and toric GIT}

\author[A. Clarke]{Andrew Clarke}
\address{Instituto de Matem\'atica, Universidade Federal do Rio de Janeiro, Av. Athos da Silveira Ramos 149, Rio de Janeiro, RJ, 21941-909, Brazil}
\email{andrew@im.ufrj.br} 

\author[C. Tipler]{Carl Tipler}
\address{Univ Brest, UMR CNRS 6205, Laboratoire de Math\'ematiques de Bretagne Atlantique, France}
\email{carl.tipler@univ-brest.fr}

\date{\today}

 \begin{abstract}
For $(X,L)$ a polarized toric variety and $G\subset \Aut(X,L)$ a torus, denote by $Y$ the GIT quotient $X/\!\!/G$. We define a family of fully faithful functors from the category of torus equivariant reflexive sheaves on $Y$ to the category of torus equivariant reflexive sheaves on $X$. We show, under a genericity assumption on $G$, that slope stability is preserved by these functors if and only if the pair $((X,L),G)$ satisfies a combinatorial criterion. As an application, when $(X,L)$ is a polarized toric orbifold of dimension $n$, we relate  stable equivariant reflexive sheaves on certain $(n-1)$-dimensional weighted projective spaces to stable equivariant reflexive sheaves on $(X,L)$.
\end{abstract}

\maketitle

 
\section{Introduction}
\label{sec:intro}
The construction of moduli spaces of projective varieties and vector bundles 
is a fundamental problem in algebraic geometry. Given a polarized variety $(X,L)$ or a vector bundle $\cE$ on $(X,L)$, one considers  various stability notions for $X$ and $\cE$ (see e.g. \cite{DonSurv,AS} for K-stability of varieties, and \cite{HuLe} for slope stability of bundles). In the presence of symmetries for $(X,L)$, that is, given an algebraic action of a reductive Lie group $G$ on $(X,L)$, it is natural to ask whether these stability notions persist on the GIT quotient $Y$ of $(X,L)$ by $G$. 
By the Yau-Tian-Donaldson conjecture \cite{DonSurv,AS} and the Kobayashi-Hitchin correspondence \cite{Ko}, the stability of $(X,L)$ or $\cE$ can be related to the existence of a canonical metric on the underlying complex object, variety or bundle. From this differential geometrical point of view, $G$-orbits can detect curvature on $X$, and canonical metrics are not necessarily preserved under GIT quotients. As a motivating case, in \cite{FutakiRicci}, Futaki investigated GIT quotients of Fano varieties, giving a condition for the symplectic reduction of $X$ to be K\"ahler-Einstein. It is then natural to expect a relation between the stability of $X$, of the quotient $Y$, and the geometric properties of the representation $G\to \Aut(X,L)$. In this paper, we provide an example of such an interplay, by relating slope stability for reflexive sheaves on $X$ and $Y$ to a combinatorial criterion on the $G$-action, in the equivariant context of toric geometry (see also \cite{GP,Mundet} for related results). 
   
A vector bundle, or more generally a torsion-free sheaf $\cE$ on a complex projective variety $X$ is said to be slope stable with respect to an ample $\R$-divisor $\alpha\in N^1(X)_\R$ if for any  subsheaf $ \cF $ with $0<\rk\,\cF<\rk\,\cE$, the  slope inequality holds
$$
\mu_\alpha(\cF)<\mu_\alpha(E),
$$
where the slope is given by the intersection theoretical formula:
$$
\mu_\alpha(\cE)=\frac{c_1(\cE)\cdot \alpha^{n-1}}{\rk\, \cE}.
$$ 
The notion of slope stability originated in the construction  of moduli spaces of sheaves \cite{HuLe}. Assume now that $(X,L)$ is a polarized toric variety over $\C$, that is endowed with an effective action of a complex torus $T_X$ with open and dense orbit. We further assume the toric varieties that we consider to come from fans, and in particular to be normal.
We denote by $N$ the lattice of one-parameter subgroups of $T_X$, so that $T_X=N\otimes_\Z\C^*$.
Consider $G\subset T_X$ a subtorus, given by a sublattice $N_0\subset N$, that is $G=N_0\otimes_\Z \C^*$.
For any linearization $\gamma:T_X \to \Aut(L)$, we can form a toric variety obtained by GIT quotient $Y=X/\!\!/G$. 
To avoid finite quotients, we will assume that $N_0$ is {\it saturated in $N$}, that is $N_0=N\cap(N_0\otimes \R)$. We will further assume  that the restriction to $G$ of the linearization $\gamma$ on $L$ is {\it generic}, which is that the stable and semi-stable loci coincide (see Section \ref{sec:toricGIT}). Under these hypothesis, we build a family of fully faithful functors 
$$\fP_\bfi : \Ref^{T_Y}(Y) \to \Ref^{T_X}(X)$$ that embeds the category of torus equivariant reflexive sheaves\footnote{Also known as {\it toric reflexive sheaves} in the literature.} on $Y$ into the category of torus equivariant reflexive sheaves on $X$ (Section \ref{sec:pullback functors}).
Given an ample class $\alpha\in N^1(Y)_\R$ on $Y$, we will say that such a functor $\fP_\bfi$ preserves slope stability notions from $(Y,\alpha)$ to $(X,L)$ if an element $\cE\in \Ref^{T_Y}(Y)$ is slope stable (resp. semistable, polystable) with respect to $\alpha$ if and only if $\fP_\bfi(\cE)$ is slope stable (resp. semistable, polystable) with respect to $L$ (see Section \ref{sec:stability notions} for the definition of these notions). Then our main result goes as follows:
\begin{theorem}
 \label{theo:intro}
 Let $(X,L)$ be a polarized toric variety with torus $T_X=N\otimes_\Z\C^*$. Let $G=N_0\otimes_\Z\C^*$ be a subtorus for a saturated sublattice $N_0\subset N$. Let $\gamma : T_X \to \Aut(L)$ be a generic linearization of $G$, and denote by $Y$ the associated GIT quotient $X/\!\!/G$. Then, the following statements are equivalent:
 \begin{itemize}
  \item[i)] There exists an ample class $\alpha\in N^1(Y)_\R $ on $Y$ such that the functors $\fP_\bfi$ preserve slope stability notions from $(Y,\alpha)$ to $(X,L)$.
  \item[ii)] The pair $((X,L), (G,\gamma))$ satisfies the {\it Minkowski condition} 
  \begin{equation}
   \label{eq:MC intro}
   \sum_{D\subset X^s} \deg_L(D)\: u_D = 0 \mod N_0\otimes_\Z \R.
  \end{equation}
 \end{itemize}
 Moreover, there is at most one class $\alpha$ on $Y$ satisfying $(i)$ up to scale.
\end{theorem}
In the statement of Theorem \ref{theo:intro}, the sum (\ref{eq:MC intro}) is over the set of $T_X$-invariant irreducible divisors in the stable locus $X^s$ of the $G$-action and $u_D$ denotes the primitive generator of the ray associated to $D$ in the fan of $X$ (see Section \ref{sec:toricdefinitions}).


The Minkowski condition (\ref{eq:MC intro}) is a very restrictive condition on $(G,\gamma)$. We will say that a subtorus $G\subset T_X$ is {\it compatible} with $(X,L)$ if there is a generic linearization $\gamma : T_X \to \Aut(L)$ for $G$ such that $((X,L),(G,\gamma))$ satisfies the Minkowski condition. We obtain in Lemma \ref{lem:restrictive action} an explicit bound, depending on the dimension and the number of rays in the fan, on the number of compatible one-parameter subgroups for polarized projective toric varieties satisfying a mild hypothesis. Nevertheless, we manage to show the following:
 \begin{proposition}
 \label{prop:orbifold to WPP intro}
 Let $(X,L)$ be a $n$-dimensional polarized toric orbifold. Denote by $m$ the number of torus fixed points of $X$.
 Up to replacing $L$ by a sufficiently high power,
 there are at least $m$ one-parameter subgroups of $T_X$ compatible with $(X,L)$ giving distinct GIT quotients. The associated GIT quotients for these subgroups are weighted projective spaces.
\end{proposition}
\begin{remark}
 By a toric orbifold we mean a toric variety with isolated quotient singularities, corresponding to a simplicial fan, as in \cite[Theorem 3.1.19 and Definition 3.1.18]{CLS}.
\end{remark}
Compact weighted projective spaces are precisely the projective toric orbifolds of Picard rank $1$, and as such are the simplest projective toric orbifolds. It is then interesting to be able to lift stable sheaves on these simpler objects to general toric orbifolds. We expect as an application of Theorem \ref{theo:intro} and Proposition \ref{prop:orbifold to WPP intro} to develop a new geometrical construction of stable bundles of low rank on toric projective varieties.

A fundamental theorem of Mehta and Ramanathan states that the restriction of a slope stable reflexive sheaf $\cE$ on $ X$ to a general complete intersection $Z\subset X$ of sufficiently high degree is again slope stable \cite{MeRa}. To the knowledge of the authors, there is no similar general statement for projections $\pi:X \to Y$, and our construction provides a result in this direction. More precisely, from Theorem \ref{theo:intro}, we deduce:
\begin{corollary}
 \label{cor:projective line bundles intro}
Let $Y$ be a projective toric variety, $(D_i)_{1\leq i\leq r}$ invariant Cartier divisors on $Y$ and $V$ the  decomposable toric vector bundle on $Y$ associated to the locally free sheaf
\begin{eqnarray*}
\cF=\cO_Y\oplus\cO_Y(D_1)\oplus\cdots\oplus\cO_Y(D_r). 
\end{eqnarray*} 
Consider the toric variety $X=\mathbb{P}(V^\vee)$, with projection map $\pi:X\to Y$. Let $L_Y$ be a polarization on $Y$ such that $L_X=\pi^*L_Y\otimes \mathcal{O}_X(1)$ is ample on $X$. Then, there exists a real ample class $\alpha\in N^1(Y)_\mathbb{R}$ such that an equivariant reflexive sheaf $\mathcal{E}$ on $(Y,\alpha)$ is slope stable if and only if $\pi^*\mathcal{E}$ is slope stable on $(X,L_X)$.
\end{corollary}
\begin{remark}
 \label{rem:ample class downstairs intro}
 In the setting of Corollary \ref{cor:projective line bundles intro}, we give examples where we can determine the class $\alpha$ on $Y$ (see Section \ref{sec:proj bundles}). This class is not the one obtained from the GIT quotient of $(X,L_X)$, that being $L_Y$ in this case. A quick look at the examples coming from Corollary \ref{cor:projective line bundles intro} suggests that in most cases, $\alpha$ will be different from $L_Y$. It would be interesting to obtain a general formula for $\alpha$ in terms of the geometric data $(X,L)$ and $G$, and in particular to understand if $\alpha$ is always rational or not.
\end{remark}
The proof of Theorem \ref{theo:intro} is divided in two main parts. The first one, in Section \ref{sec:toricdescent}, is the construction of the functors $\fP_\bfi$. It is naturally associated to the study of the descent of equivariant reflexive sheaves on a toric variety $X$ under a generic toric GIT quotient $X \dashrightarrow Y$. Let us denote by $\iota : X^s \to X$ the inclusion of the stable locus and by $\pi : X^s \to Y$ the projection to the quotient. An equivariant sheaf $\cE$ descends to $Y$ if there exists a sheaf $\check \cE$ on $ Y$ such that $\pi^*\check \cE$ is equivariantly isomorphic to $\iota^*\cE$. In \cite{Ne}, Nevins gave a general criterion for the descent of a sheaf through a good quotient. In Section \ref{sec:Nevins reflexive}, we give a combinatorial criterion for the descent of equivariant reflexive sheaves under generic toric GIT quotients. We then build the functors $\fP_\bfi$, by extending equivariant reflexive sheaves pulled-back from $Y$ to $X^s$ across the unstable locus (see \cite{IlSu} for similar constructions). The elements in the images of these functors are described geometrically, and correspond precisely to the reflexive sheaves that descend to $Y$ and for which the slopes on $X$ and $Y$ will be comparable.
The second part of the proof of Theorem \ref{theo:intro}, given in Section \ref{sec:slopes and minkowski}, gives a relation between slopes on $(X,L)$ and slopes on $(Y,\alpha)$. A combinatorial formula describes these slopes \cite{HNS,DDK}. In this formula, there are contributions from the sheaves and from the polarizations. The functors $\fP_\bfi$ are precisely constructed so that the sheaf contributions can be compared on $X$ and $Y$. As for the polarization terms, they are related to the volumes of the facets of the associated polytopes. To be able to compare them through the quotient, we use a classical result of Minkowski stating precisely when the volumes of the facets of a polytope can be prescribed.

\begin{remark}
It might be worth noting that part of our results are not inherent to the toric world and should admit more general formulations. Given a polarized variety $(X,L)$ with a linearized action of a reductive group $G$ and with GIT quotient $Y$, one should be able to construct functors $\fP_\bfi$ from the category of reflexive sheaves on $Y$ to the category of $G$-equivariant reflexive sheaves on $X$ as we do in Section \ref{sec:pullback functors}. 
However, being able to determine a class $\alpha\in N^1(Y)_\mathbb{R}$ such that the slopes of the appropriate sheaves on $X$ and $Y$ can be compared seems to be a much more delicate problem in such generality.
This should be possible in some specific situations though, and we hope to obtain generalizations of Theorem \ref{theo:intro} for varieties with large symmetry groups, such as $T$-varieties of low complexity or spherical varieties. 
\end{remark}

The organization of the paper is as follows. In Section \ref{sec:toricdefinitions} we gather standard facts about toric varieties and equivariant reflexive sheaves that will be used in the paper. In particular, we recall in Section \ref{sec:toricpolarized} the classical correspondence between polytopes and polarized toric varieties, and describe its generalization to real ample divisors. In Section \ref{sec:torictorsionfreesheaves}, we recall Klyachko's description of the category of equivariant reflexive sheaves on toric varieties. Along the way, we give a new and shorter proof for the combinatorial formula for the first Chern class of these objects, extending several earlier results to normal toric varieties (compare with \cite{Koo11,HNS,DDK}). Section \ref{sec:toricdescent} deals with the descent of reflexive sheaves under toric GIT. We start by recalling the necessary material of toric GIT in Section \ref{sec:toricGIT}, then prove a descent criterion in Section \ref{sec:Nevins reflexive}, and last construct the pullback functors in Section \ref{sec:pullback functors}. With this material at hand, we can prove Theorem \ref{theo:intro} in Section \ref{sec:slopes and minkowski}. Together with Section \ref{sec:toricdescent}, they form the core of the paper. We first introduce the notions of slope stability in Section \ref{sec:stability notions}, and then recall a classical theorem of Minkowski in Section \ref{sec:Minkowski condition}, to conclude with the proof of our main theorem. Finally, in Section \ref{sec:applications} we study compatible actions and give applications of our result, proving Proposition \ref{prop:orbifold to WPP intro} and Corollary \ref{cor:projective line bundles intro}. 

\subsection*{Acknowledgments}  
We would like to thank Alberto Della Vedova, Henri Guenancia, Johannes Huisman, \'Eveline Legendre, Yann Rollin and Hendrik S\"uss for stimulating discussions. We are also grateful to the referees for useful comments that helped in improving the paper. The first named author would like to acknowledge the financial support of the CNRS and CAPES-COFECUB that made possible his visit to LMBA.
 

\section{Equivariant reflexive sheaves on toric varieties}
\label{sec:toricdefinitions}
Throughout this paper, we consider toric varieties over the complex numbers.
We recall the description of polarized toric varieties in terms of polytopes and the characterization of equivariant reflexive sheaves on toric varieties in terms of families of filtrations. 

\subsection{Polarized toric varieties and polytopes}
\label{sec:toricpolarized}

We refer to \cite[Chapters 2, 3, 6]{CLS} and \cite{thad} for this section. Let $N$ be a rank $n$ lattice,
$M:=\Hom_\Z(N,\Z)$ be its dual with pairing $\langle\cdot,\cdot\rangle$. Then $N$ is the lattice
of $1$-parameter subgroups of a $n$-dimensional complex torus $$T:=N\otimes_\Z \C^*= \Hom_\Z(M,\C^*).$$
We set $N_\F=N\otimes_\Z \F$ and $M_\F=M\otimes_\Z \F$ for
$\F=\Q$ or $\R$. 

Let $X=X(\Sigma)$ be an $n$-dimensional projective toric variety associated to a fan $\Sigma$, so that in particular $X$ is normal. Denote $\Sigma=\lbrace \sigma_i\, : i\in I \rbrace$, with $\sigma_i$ a strongly convex rational polyhedral cone in $N_\R$ for all $i\in I$. Denote also by $\Sigma(k)$ the set of $k$-dimensional cones in $\Sigma$.
The variety $X$ is obtained by gluing affine charts $(U_\sigma)_{\sigma\in\Sigma}$, with
$$
U_\sigma=\Spec(\C[S_\sigma]),
$$
and $\C[S_\sigma]$ is the semi-group algebra of
$$
S_\sigma=\sigma^\vee\cap M
=\lbrace m\in M:   \langle m , n \rangle \geq 0 \text{ for all }n\in \sigma\rbrace .
$$

There exists a bijective correspondence between cones $\sigma\in\Sigma$ and $T$-orbits $O(\sigma)$ in $X$. This satisfies, for $\sigma\in \Sigma$, $\dim O(\sigma)=n-\dim(\sigma)$.  In particular, to any $\rho\in\Sigma(1)$, there corresponds a irreducible $T$-invariant Weil divisor $D_\rho$ given by 
\begin{eqnarray}\label{eq:orbitcone correspondence}
D_\rho=\overline{O(\rho)} 
\end{eqnarray}
where the closure is in both classical and Zariski topologies.

As $X$ is associated to the fan $\Sigma$, there is a bijective correspondence between $T$-invariant ample divisors on $X$ and lattice polytopes $P\subset M_\R$ whose normal fan $\Sigma_P$ is equal to $\Sigma$ (see \cite[Theorem 6.2.1]{CLS}). 
Let $D$ be a $T$-invariant Cartier divisor on $X$.
Recall that it is equal to a linear combination of the form 
$$
D =\sum_{\rho \in \Sigma{(1)}} a_\rho D_\rho.
$$
For each $\rho\in \Sigma {(1)}$ we denote by
$u_\rho\in N$ the minimal generator of $\rho\cap N$. Assuming that $D$ is ample,
we consider the associated polytope $P=P_D\subset M_\R$:
\begin{eqnarray}
P =  \{m\in M_\mathbb{R}\ : \ \langle m,u_\rho\rangle \geq -a_\rho \text{ for all } \rho\in \Sigma(1)\} \label{eqn:polytope def}.
\end{eqnarray}
Note that if $D^\prime$ is equivariant and linearly equivalent to $D$ then $P_{D^\prime}$ is given by translation of $P_D$ in $M$ by some lattice element $m\in M$.  In the same way, a lattice translation of the polytope $P_D$ corresponds to a different linearization of the action of $T$ on the line bundle $\cO(D)$ (see Section \ref{sec:toricGIT}).

As $D$ is ample, and $P$ has normal fan equal to $\Sigma$, we have a correspondence between cones in $\Sigma$ and faces of $P$. For a face $Q$ in $P$, we denote by $\sigma_Q\in\Sigma$ (resp. by $O(Q)$) the associated cone (resp. the associated orbit). In particular, rays in $\Sigma$ corresponds to facets of $P$. For each $\rho\in\Sigma(1)$ the associated facet is
\begin{eqnarray*}  
F =  P\cap   \{m\in M_\mathbb{R}\ : \ \langle m,u_\rho\rangle = -a_\rho\}.
\end{eqnarray*} 
We will denote
$u_\rho$ by $u_F$ and  $a_\rho$ by $a_F$. 
We can also write 
\begin{equation*}
D =\sum_{F \prec P} a_F D_F,
\end{equation*}
where the sum is over all facets of $P$ and $D_F:=D_{\rho_F}$. We will denote faces  of $P$ of higher codimension by the letter $Q$ and vertices by the letter $v$. We use the relation $Q_1\preccurlyeq Q_2$ to signify that $Q_1$ and $Q_2$ are faces of $P$,  possibly equal to $P$ itself,  and $Q_1\subseteq Q_2$.

The correspondence between polarizations on $X$ and lattice polytopes with normal fan $\Sigma$ modulo lattice translations extends to {\it real} ample classes.
We recall the definition of real ample divisors on a normal complex algebraic variety $X$ (see \cite{La}). Let $\Div(X)$ denote the group of integral Cartier divisors on $X$. The N\'eron-Severi group is given by $N^1(X) =\Div(X)/\sim_{num}$  and we set $N^1(X)_\mathbb{R}=N^1(X)\otimes_\mathbb{Z}\mathbb{R}$. Denote by $\WDiv(X)$ the set of Weil divisors on $X$ and $\WDiv_\mathbb{R}(X) $ the vector space of real Weil divisors. 
\begin{definition}
 \label{def:ample class}
A class $\alpha\in N^1(X)_\mathbb{R}$ is \emph{ample} if it is represented by a positive real linear combination of ample Cartier divisors.
\end{definition}
We then have the following:
\begin{proposition}
 \label{prop:real ample class correspondence}
 Let $X$ be a projective toric variety given by a fan $\Sigma$. 
 Then there is a bijective correspondence between real ample classes on $X$ and real polytopes of the form
 $$
 P=\lbrace m\in M_\R\ :\  \langle m, u_\rho \rangle \geq - a_\rho,\:\text{ for all } \rho\in \Sigma(1)\rbrace 
 $$
for which the normal fan $\Sigma_P=\Sigma$,  
 modulo real translations within $M_\mathbb{R}$. 
\end{proposition}

\begin{proof} 
This statement between integral classes and lattice polytopes is standard (note that for projective toric varieties coming from fans, the real Picard group and the real N\'eron-Severi group coincide, see \cite[Proposition 6.3.15]{CLS}). The rational case follows by clearing denominators and scaling polytopes. We now prove the real case. Set $\mathrm{Pic}(X)$ and $\mathrm{Cl}(X)$ the Picard and class groups of $X$. From the exact sequence (see \cite[Theorem 4.1.3]{CLS}):
\begin{eqnarray*}
0\longrightarrow M\longrightarrow \bigoplus_{\rho\in\Sigma(1)}\mathbb{Z}\cdot D_\rho\stackrel{\pi}{\longrightarrow}\mathrm{Cl}(X)\longrightarrow 0
\end{eqnarray*}
we deduce the sequences of vector spaces, for $\K=\Q$ or $\R$:
\begin{eqnarray*}
0 \longrightarrow M_\mathbb{K}\longrightarrow W_\mathbb{K}\longrightarrow \mathrm{Pic}(X)_\mathbb{K}\longrightarrow 0
\end{eqnarray*}
where we set $W=\pi^{-1}(\mathrm{Pic}(X))$, $W_\K=W\otimes_\mathbb{Z}\mathbb{K}$ and $\mathrm{Pic}(X)_\mathbb{K}=\mathrm{Pic}(X)\otimes_\mathbb{Z}\mathbb{K}$. 

Let $\alpha\in \mathrm{Pic}(X)_\mathbb{R}= N^1(X)_\mathbb{R}$ be an ample class. We can represent $\alpha$ by 
$$D=\sum_{\rho\in\Sigma(1)} a_\rho D_\rho\in W_\R.$$ 
Define the set 
$$
P_\alpha= \{m\in M_\mathbb{R}\ :\ \forall\rho\in\Sigma(1),\ \langle m,u_\rho\rangle\geq -a_\rho\}.
$$
First observe that $P_\alpha$ is a polytope, rather than a polyhedron, since the fan $\Sigma$ is complete. By definition of ample real divisors, $D=\sum_{i=1}^N\alpha_iD_i$ for $D_i$ ample Cartier divisors and $\alpha_i$ positive real numbers. 
By the openness of the ampleness condition, we can take positive rational numbers $\beta_i\in\mathbb{Q}$ as close as we like to the $\alpha_i\in\mathbb{R}$ in such a way that 
\begin{eqnarray*}
D_\beta=\sum_{i=1}^N \beta_i D_i=\sum_{\rho\in\Sigma(1)} b_\rho D_\rho
\end{eqnarray*}
is an ample $\mathbb{Q}$-divisor, and hence the polytope 
\begin{eqnarray*}
P_{\beta}=\{m\in M_\R\ :\forall\rho\in\Sigma(1),\ \langle m,u_\rho\rangle \geq -b_\rho\}
\end{eqnarray*}
has normal cone $\Sigma_{P_\beta}=\Sigma$.
Furthermore, decomposing the divisors $D_i$ in the basis $\lbrace D_\rho\, :\, \rho\in\Sigma(1) \rbrace$, we see that the values $(b_\rho)$ vary continuously with respect to $(\beta_i)$, so can be made sufficiently close to the $(a_\rho)$ to guarantee that $\Sigma_{P_\alpha}=\Sigma_{P_\beta}=\Sigma$. Note that as in the integral case, translations of $P_\alpha$ by elements of $M_\R$ correspond to different choices of representative of the class $\alpha$ in $W_\R$.
 
For the converse statement, we consider the polytope in $M_\mathbb{R}$
\begin{eqnarray*}
P=\{m\in M_\mathbb{R}\,:\,\forall\rho\in\Sigma(1), \langle m,u_\rho\rangle \geq -a_\rho\}
\end{eqnarray*}
for $a_\rho\in \mathbb{R}$, supposing that the normal fan of $P$ determined by the vectors $u_\rho$ is the fan of $X$. 
Then, the polytope $P$ determines the real Weil divisor $D_P=\sum_\rho a_\rho D_\rho\in \WDiv_\mathbb{R}(X)$. 
We show that $D_P$ lies in the space of real Cartier divisors, and is moreover ample. This is proven by induction on the number of $a_\rho$'s that are irrational. We list the rays in $\Sigma$ by $\rho_1,\ldots,\rho_d$ for $d=\#\Sigma(1)$. As noted above, the case where all $a_\rho$'s are rational is well-known. Suppose that, for fixed $k\geq 1$, $D=\sum_{i=1}^d a_{\rho_i}D_{\rho_i}$ defines a real ample class whenever its normal fan $\Sigma_P=\Sigma$ and  $a_{\rho_i}\in\mathbb{Q}$ for all $i\geq k$.  If $a_{\rho_i}\in\mathbb{Q}$ for $i\geq k+1$ then let $r_1,r_2$ be rational numbers for which $r_1<a_{\rho_k}<r_2$ sufficiently close to $a_{\rho_k}$ that for any $s\in [r_1,r_2]$ the polytope defined by the inequalities 
\begin{eqnarray*}
\langle m,u_{\rho_i}\rangle &\geq & -a_{\rho_i},\ \ \text{ for } i\neq k,\\
\langle m,u_{\rho_k} \rangle &\geq & -s
\end{eqnarray*}
defines the same normal fan $\Sigma$. Then, for some $t\in [0,1]$, we have $a_{\rho_k}=tr_1+(1-t)r_2$ and 
\begin{eqnarray*}
D=\sum_{i=1}^da_{\rho_i}D_{\rho_i} &=& t\left( r_1 D_{\rho_k}+\sum_{i\neq k} a_{\rho_i}D_{\rho_i} \right) +(1-t)\left( r_2D_{\rho_k}+\sum_{i\neq k}a_{\rho_i}D_{\rho_i}\right).
\end{eqnarray*}
By the induction hypothesis, each of the two real divisors on the right hand side of the above equality is ample. By the convexity of the set of real ample classes, $D$ is ample. 
\end{proof}

\begin{remark}
We note that in the smooth case, a similar result can be given via symplectic geometry by using the correspondence between compact symplectic toric manifolds and Delzant polytopes up to translations.
\end{remark}

\subsection{Equivariant reflexive sheaves}
\label{sec:torictorsionfreesheaves}
We refer to \cite{Kl,Koo11,perl} for this section.
We consider a projective toric variety $X$ together with a polytope $P\subset M_\R$  associated to an ample class on $X$. We are interested in a combinatorial description of equivariant sheaves on $X$.
\begin{remark}
 From now on, by equivariant sheaf on $X$ we will mean equivariant with respect to the torus $T$ of $X$. When later we  consider equivariant sheaves on $Y$ a GIT quotient of $X$, we will mean equivariant with respect to the torus of $Y$.
\end{remark}
Recall that a reflexive sheaf on $X$ is a coherent sheaf $\cE$ that is canonically isomorphic to its double dual $\cE^{\vee\vee}$.
Klyachko gave a description of equivariant reflexive sheaves in terms of combinatorial data:
\begin{definition}
 \label{def:family of filtrations}
 A family of filtrations $\E$ is the data of a finite dimensional vector space $E$ and for each facet $F$ of $P$, an increasing filtration $(E^F(i))_{i\in\Z}$ of $E$ such that $E^F(i)=\lbrace 0 \rbrace$ for $i\ll 0$ and $E^F(i)=E$ for some $i$. 
We will denote by $i_F$ the smallest $i\in \Z$ such that $E^F(i)\neq 0$.
 \end{definition}
 \begin{remark}
 Families of filtrations in \cite{Kl} or \cite{perl} are labeled by the set of rays $\rho\in\Sigma(1)$. As $P$ is associated to an ample class, there is a one-to-one correspondence between the facets of $P$ and the rays of the fan of $X$, and so we recover the usual definition.
 Note also that we  use increasing filtrations of $E$, as in \cite{perl}, rather than decreasing as in \cite{Kl}.
\end{remark}
To a family of filtrations $\E:=\lbrace (E^F(i))\subseteq E\,:\, F\prec P, i\in \Z\rbrace $ we can assign an equivariant reflexive sheaf $\cE:=\fK(\E)$ defined by 
 \begin{equation}
  \label{eq:sheaf from family of filtrations}
  \Gamma(U_{\sigma_Q}, \cE):=\bigoplus_{m\in M} \bigcap_{Q\preccurlyeq F} E^F(\langle m,u_F\rangle)\otimes \chi^m
 \end{equation}
 for all proper faces $Q\prec P$, while $\Gamma(U_{\sigma_P},\cE)=E\otimes \mathbb{C}[M]$. 
 \begin{remark}
The conditions for a family of filtrations to define a locally-free sheaf are determined in \cite{Kl}.
 \end{remark}
 The morphisms of families of filtrations are defined by:
 \begin{definition}
  \label{def:morphism of family of filtrations}
  A morphism between two families of filtrations $\E_1=\lbrace (E_1^F(i))\subseteq E_1: F\prec P, i\in \Z\rbrace$ 
  and $\E_2=\lbrace (E_2^F(i))\subset E_2:F\prec P, i\in \Z\rbrace$ is a linear map $\phi : E_1 \to E_2$ preserving the filtrations, that is such that for all $F$ and all $i$, $\phi(E_1^F(i)) \subset E_2^F(i)$.
 \end{definition}
 Such a morphism between families of filtrations induces an equivariant morphism between the associated reflexive sheaves. 
 \begin{remark}
  In this paper, by a morphism between equivariant reflexive sheaves $\cE_1=\fK(\E_1)$ and $\cE_2=\fK(\E_2)$, we mean an equivariant morphism of degree zero: the space $\mathrm{Hom}(\cE_1,\cE_2)$ of all sheaves homomorphisms is naturally graded over $M$:
  $$
  \mathrm{Hom}(\cE_1,\cE_2)\subset \mathrm{Hom}(E_1,E_2)\otimes\C[M]
  $$
  and the morphisms that we will consider are those lying in the $\mathrm{Hom}(E_1,E_2)\otimes \chi^0$-component of the grading. In particular, this excludes injective maps like $\cO_{\P^1}(-1)\to \cO_{\P^1}$.
 \end{remark}

For a toric variety $Z$, and an ample polytope $P_Z$,
we denote by:
\begin{enumerate}
 \item[i)] $\Ref^T(Z)$ the category of equivariant reflexive sheaves on $Z$,
 \item[ii)] $\Filt(P_Z)$ the category of families of filtrations associated to $P_Z$.
\end{enumerate}
 From Klyachko and Perling \cite{Kl,perl}, we obtain the following:
 \begin{theorem}{\rm \cite{Kl,perl}}
 The functor $\fK$ induces an equivalence of categories between the category $\Filt(P)$ and the category $\Ref^T(X)$.
 \end{theorem}
 As the category of filtrations on a given finite dimensional vector space is  abelian, we have:
\begin{corollary}
 \label{cor:equiv ref sheaves abelian category}
 The category $\Ref^T(X)$ of equivariant reflexive sheaves on $X$ is an abelian category.
\end{corollary}
 We will need the combinatorial characterizations of equivariant reflexive saturated subsheaves and of equivariant rank $1$ reflexive sheaves. Let $\cE=\fK(\E)$ be an equivariant reflexive sheaf on $X$, given by a family of filtrations $\E=\lbrace (E^F(i))\subset E: F\prec P, i\in\Z\rbrace$. For any vector subspace $W\subset E$, define a family of filtrations $\E\cap W$ by
  $$
  \E\cap W = \lbrace  (W\cap E^F(i))\subset W\cap E: F\prec P, i\in\Z\rbrace.
  $$
  Then, the sheaf $\cE_W:=\fK(\E\cap W)$ is an equivariant saturated reflexive subsheaf of $\cE$. Any equivariant reflexive saturated subsheaf of $\cE$ arises that way:
 \begin{proposition}{\rm (\cite[Rem. 2.4.]{HNS})}
 \label{prop:equiv ref subsheaves}
 Let $\cE=\fK(\E)$ be an equivariant reflexive sheaf on $X$. Let $\cF\subset \cE$ be an equivariant saturated reflexive subsheaf of $\cE$. Then, there is a unique vector subspace $W\subset E$ such that $\cF=\fK(\E\cap W)$.
 \end{proposition}
 As for rank $1$ reflexive sheaves, from the definition we obtain:
 \begin{proposition}
  \label{prop:rank one}
  Let $\cO(-D)$ be the rank $1$ reflexive sheaf associated to the invariant Weil divisor $D=\sum_{F\prec P} a_F D_F$. Then, $\cO(-D)=\fK(\E_D)$, where the family of filtrations $\E_D=\lbrace (E^F(i))\subset \C, F\prec P, i\in\Z\rbrace$ satisfies
\begin{equation*}
E^F(i)=
\left\{ 
\begin{array}{ccc}
0 & \mathrm{ if } & i < a_F\\
\C & \mathrm{ if } & i \geq a_F.
\end{array}
\right.
\end{equation*}
 \end{proposition}
We will also need the determinant and first Chern class of reflexive sheaves. 
\begin{remark}
Let $A_k(X)$ be the $k$-th Chow group of $X$. This is the quotient of the free abelian group on $k$-dimensional subvarieties by rational equivalence. The first Chern class is the map $c_1:\mathrm{Pic}(X)\to A_{n-1}(X)$ induced by the inclusion of Cartier divisors in Weil divisors. This defines a morphism $A_k(X)\to A_{k-1}(X)$ for each $k$ as follows. For $\mathcal{L}$ a line bundle and $V$ a $k$-dimensional subvariety on $X$, $\mathcal{L}|_V$ defines a Cartier divisor on $V$, hence a $(k-1)$-cycle on $X$. 

As $X$ is normal, this definition extends to rank one reflexive sheaves since every rank-one reflexive sheaf is of the form $\mathcal{E}=\mathcal{O}_X(D_\mathcal{E})$ for some Weil divisor $D_\mathcal{E}$, and hence $c_1(\mathcal{E}):=[D_\mathcal{E}]\in A_{n-1}(X)$. If $H$ is an ample line bundle on $X$, the degree of $\mathcal{E}$ is then given by 
\begin{eqnarray*}
\deg_H(\mathcal{E})= c_1(\mathcal{E})\cdot H^{n-1}\in A_0(X)\cong \mathbb{Z}.
\end{eqnarray*}
\end{remark}
Recall the following:
\begin{definition}
 \label{def:det and first chern class}
If $\cE$ is a torsion-free coherent sheaf on $X$, one defines the determinant of $\cE$ to be the rank-one reflexive sheaf $\det(\cE)=(\Lambda^{\rk(\cE)}\cE)^{\vee\vee}$. Then, the first Chern class of $\cE$ is $c_1(\cE):=c_1(\det \cE)$.
\end{definition}

We prove now that the previously known formula for the first Chern class for equivariant reflexive sheaves holds in our case, as a linear combination of Weil divisors. This extends the previous case on smooth varieties, given by Kool, to normal varieties. To produce a necessary combinatorial formula, we first need to introduce some notation:
\begin{definition}
 \label{def:defining data}
Let $\cE=\fK(\E)$, with $\E=\lbrace (E^F(i))\subset E :F\prec P, i\in \Z \rbrace$. We set, for all $F\prec P$ and all $i\in\Z$:
\begin{equation}
 \label{eq:defining ei}
 e^F(i)=\dim(E^F(i))-\dim (E^F(i-1)).
\end{equation}
We will refer to the integers $(e^F(i))_{F\prec P, i\in \Z}$
as the {\it dimension jumps} of $\E$ or $\cE$.
\end{definition}
Then we have:
\begin{proposition}
 \label{prop:determinant sheaf}
 Let $\cE=\fK(\E)$ be a rank $r$ equivariant reflexive sheaf on $X$, given by a family of filtrations $\E=\lbrace (E^F(i))\subset E: F\prec P, i\in\Z\rbrace$.
 We define a family of filtrations $\det(\E)=\lbrace (E_{\det}^F(i))\subset \Lambda^r E: F\prec P, i\in\Z\rbrace$ by setting, for all $F\prec P$,
\begin{equation*}
E_{\det}^F(i)=
\left\{ 
\begin{array}{ccc}
0 & \mathrm{ if } & i < i_F(\det\cE)\\
\Lambda^r E & \mathrm{ if } & i \geq i_F(\det\cE)
\end{array}
\right.
\end{equation*}
 where for all $F\prec P$,
 $$
 i_F(\det\cE)=\sum_{i\in\Z} i e^F(i).
 $$
 Then
$\det(\cE)=\fK(\det(\E))$.
\end{proposition}
  \begin{proof}
 Note first that because $\Ref^T(X)$ is abelian, $\Lambda^r\cE$ is reflexive and $\det(\cE)=\Lambda^r \cE$. Then, the family of filtrations $\lbrace (\Lambda^r E^F(i)): F\prec P, i\in \Z \rbrace$ for $\Lambda^r \cE$ satisfies:
 $$
 \Lambda^r E^F(i) = \sum_{i_1,\ldots, i_r: \sum i_j = i} E^F(i_1) \wedge \ldots \wedge E^F(i_r).
 $$
 Now, $\Lambda^r E$ is one dimensional, so  
 \begin{equation*}
\Lambda^r E^F(i)=
\left\{ 
\begin{array}{ccc}  
0 & \mathrm{ if } & i < j_F\\
\Lambda^r E & \mathrm{ if } & i \geq j_F
\end{array}
\right.
\end{equation*}
where $j_F$ is the smallest integer $l\in \Z$ such that there is a partition $i_1,\ldots,i_r$ of $l$ with $E^F(i_1) \wedge \ldots \wedge E^F(i_r)\neq \lbrace 0 \rbrace$. From the fact that $(E^F(i))$ forms a filtration of vector spaces, we deduce that $j_F$ must be the sum of the integers $i$ such that the dimension of $E^F(i)$ changes, counted with multiplicity. Hence $j_F = \sum_{i\in\Z} i e^F(i)$, which ends the proof.
\end{proof}
\begin{corollary}
 \label{prop:first Chern class}
 Let $\cE=\fK(\E)$ be an equivariant reflexive sheaf on $X$, given by a family of filtrations $\E=\lbrace (E^F(i))\subset E:F\prec P, i\in\Z\rbrace$.
 The first Chern class of $\cE$ is the class of the Weil divisor:
\begin{equation}
 \label{eq:first Chern class}
 c_1(\cE)=-\sum_{F\prec P} i_F(\det\cE) \: D_F.
\end{equation}
 where for all $F\prec P$,
 $$
 i_F(\det\cE)=\sum_{i\in\Z} i e^F(i).
 $$
\end{corollary}

 \begin{remark}
  In this paper, we restrict ourselves to reflexive sheaves. We expect that most of the results extend to equivariant torsion-free coherent sheaves, described in terms of families of multifiltrations \cite{perl,Koo11}. Since the applications we have in mind concern stable vector bundles, it is enough to consider the category of reflexive sheaves, where the results and proofs are simpler to express. 
 \end{remark}
 

\section{Descent of equivariant sheaves under toric GIT}
\label{sec:toricdescent}
In this section we study the descent of equivariant reflexive sheaves under toric GIT quotients. We denote by $X$ a projective toric variety, polarized by an ample equivariant line bundle $L$. We retain the notation of the previous section.
\subsection{Toric GIT}
\label{sec:toricGIT}
We refer to \cite{MumFoKi,thad} for this section. We are interested in GIT quotients of $(X,L)$ by subtorus actions. Let $N_0$ be a sublattice of $N$ of rank $g$. We will assume that $N_0$ is saturated, that is $N_0=(N_0\otimes_\Z \R)\cap N$. The sublattice $N_0$ spans a $g$-dimensional subtorus $G=N_0\otimes_\Z \C^*$ of $T$. We fix a linearization $\gamma$ of $T$ on $L$ and
we will consider the GIT quotient with respect to the induced linearization of $G$. From Section \ref{sec:toricpolarized}, $(X,L)$ defines a family of lattice polytopes $\lbrace P_D:\cO(D)\sim L\rbrace$, all equal up to translations by lattice elements. Then, the linearization $\gamma$ determines a unique polytope $P$ in this family such that
\begin{equation}
 \label{eq:linearisation sections}
 H^0(X,L)=\bigoplus_{m\in P\cap M} \C\cdot \chi^m
\end{equation}
is the weight decomposition of the $T$-action on the space of global sections induced by $\gamma$ (see \cite[Proposition 4.3.3]{CLS}).
We have the following:
$$
0 \to N_0 \to N \to N/N_0 \to 0,
$$
and the associated dual sequence:
$$
0 \to N_0^\perp \to M \to M/N_0^\perp \to 0.
$$
Set $U=N_0^\perp\otimes_\Z \R\subset M_\R$.
Then from \cite[Proposition 3.2]{thad}:
\begin{proposition}
 The GIT quotient of $(X,L)$ by $G$
 is the polarized toric variety $(Y,\check{L})$
 described by the polytope $P_Y=U\cap P$. Moreover, its lattice is $N_Y=N/N_0$ with dual $N_0^\perp$. 
\end{proposition}
\begin{remark}
Note that the vertices of the polytope $P\cap U$ are not necessarily lattice points, so this polytope only induces a \emph{rational} polarization $\check L$ on the variety $Y=X/\!\!/ G$.
\end{remark}
\begin{remark}
 Up to composition by a finite morphism, we can always reduce to the case $N_0$ saturated. In this paper, we will restrict ourselves to this simpler situation.
\end{remark}
We can also describe the stable and semistable points
:
\begin{proposition}[\cite{thad}, Lemma 3.3]
\label{prop:thadXsandXss}
 The semistable and stable loci $X^{ss}$ and $X^s$ under the $G$ action on $(X,L)$ are each unions of $T$-orbits. More precisely, given a face $Q\preccurlyeq P$, the orbit $O(Q)$ is:
 \begin{itemize}
  \item semistable iff $Q\cap U \neq \varnothing$,
  \item stable iff $Q\cap U \neq \varnothing$ and the interior of $Q$ meets $U$ transversally.
 \end{itemize} 
\end{proposition}
From this proposition, the set of faces of the polytope of $X^{ss}/\!\!/G$ is 
$$
\lbrace U\cap Q\, :\, Q\preccurlyeq P\rbrace.
$$
We will denote by $P^s$ (resp. $P^{ss}$ and $P^{us}$) the set of faces corresponding to stable orbits (resp. semistable orbits and unstable orbits) of $X$.
For technical reasons, we will need the following assumption on the action:
\begin{definition}
 \label{def:generic Cstar}
 A pair of a subtorus $G\subset T$ and linearization $\gamma:T\to \Aut(L)$ will be called {\it generic} if the stable and semi-stable loci of the $G$ action on $(X,L)$ coincide and are not empty.
\end{definition}
The nice fact about generic actions is the following:
\begin{lemma}
 \label{lem:genericity lemma}
 Assume that $(G,\gamma)$ is generic. Then:
 \begin{enumerate}
  \item[i)] The points of $X^s$ have finite stabilizers under the $G$-action.
  \item[ii)] There is a $1:1$-correspondence between facets of $P_Y$ and facets in $P^s$.
 \end{enumerate}
\end{lemma}
\begin{proof}
 The first point follows by definition of stable points. For the second point, any facet of $U\cap P$ must be of the form $U\cap Q$ for some face $Q$ of $P$. By assumption, $Q$ meets $U$ transversally. Moreover, $\dim(U\cap Q)=\dim(Y)-1=n-1-g$, where $g$ is the rank of the sub-lattice $N_0\subseteq N$ and equal to the dimension of the subgroup $G$. This forces the dimension of $Q$ to be $n-1$, hence the result.
\end{proof}
From now on, unless explicitly stated, we will assume that $(G,\gamma)$ is generic.
We end this section with some definitions and lemmas relating the primitive normals to the facets of $P$ to those of $P_Y$.
We consider the lattice projection, $\pi:N\to N_Y=N/N_0$.
\begin{definition}
 For each facet $F\cap U$ of $P_Y$, set $\check u_F$ to be the primitive element in $N_Y$ defining $F\cap U$ as an inner pointing normal.
 \end{definition}
 \begin{lemma}
 \label{lem:projection primitive defining facet}
 There exists a unique positive integer $b_F$ such that 
 $$
 \check u_F=\frac{1}{b_F} \pi(u_F).
 $$
\end{lemma}
\begin{proof}
 Both elements lies on the ray $\rho_{F\cap U}$, and $\check u_F$ generates this ray.
\end{proof}
Let $F$ be facet in $P^s$. As $u_F$ is primitive, we can complete it into a basis $\cB_F:=\lbrace u_F,u_i : \: i= 2\ldots n \rbrace$ of $N$. We denote by $\cB_F^*=\lbrace m_F, m_i : \: i=2\ldots n  \rbrace$ the dual basis of $M$. In particular, note that $\lbrace m_i:\: i=2\ldots n \rbrace$ is a basis for $u_F^\perp\cap M$. Similarly, we set $\check\cB_F:= \lbrace \check u_F, \check u_i\,: \: i= 2\ldots n-g \rbrace$ a basis for $N_Y$ with dual basis $\check\cB_F^*:= \lbrace \check m_F, \check m_i: \: i= 2\ldots n-g \rbrace$ and again $\lbrace \check m_i:\: i=2\ldots n-g \rbrace$ is a basis for $\check u_F^\perp\cap N_Y$.
From Lemma \ref{lem:projection primitive defining facet}, we deduce:
\begin{lemma}
\label{lem:relations of m_F}
The element $\check m_F\in N_0^\perp\subset M$ can be uniquely written as
$$
\check m_F= b_F m_F + m_F^\perp
$$
with $m_F^\perp\in u_F^\perp$.
\end{lemma}
We conclude this section with the following observation:
\begin{remark}
 If $X$ is smooth, it is easy to show in a local chart $U_F=\Spec(\C[\sigma_F^\vee\cap M])$ that $b_F$ is the order of the stabilizer of the orbit $\cO(F)$ under the $G$-action.
\end{remark}

\subsection{Descent criteria for equivariant reflexive sheaves}
\label{sec:Nevins reflexive}
We retain the notation from the previous section.
We assume as before that $(X,L)$ is a polarized toric variety of dimension $n$ and that $(G,\gamma)$ is generic. We consider $(Y,\check L)$ the GIT quotient associated to the $G$ action on $(X,L)$. Following \cite{thad} and the discussion in the previous section, we recall that $Y$ is a simplicial toric variety, with torus given by $T_Y=T_X/G$, and $\check L$ is the invariant ample line bundle on $Y$ corresponding to the polytope $P_Y=P\cap U$. We wish to compare equivariant reflexive sheaves on $X$ and equivariant reflexive sheaves on $Y$.
We denote by $\iota : X^s \to X$ the inclusion and by $\pi: X^s \to Y$ the projection. We then introduce:
\begin{definition}
 \label{def:descent}
 We say that a $G$-equivariant coherent sheaf $\cE$ on $X$ descends to $Y$ if there is a coherent sheaf $\cF$ on $Y$ such that 
 $ \pi^*\cF$ is $G$-equivariantly isomorphic to $\iota^* \cE$.
\end{definition}
Let $\cE$ be a $G$-equivariant coherent sheaf on $X$.
For simplicity we will denote by $\cE^s$ the restriction $\iota^*\cE$.
As explained for example in \cite{Ne}, $\cE$ descends to $Y$ if and only if 
$$
\cE^s\simeq \pi^* \pi^G_* \cE^s
$$
where $\pi^G_*$ is the invariant pushforward.
\begin{remark}
Note that the functors $\iota^*$, $\pi^*$ and $\pi^G_*$ preserve the torus equivariantness and reflexivity properties of coherent sheaves (for $\pi^*$, it follows from flatness of $\pi$).
\end{remark}
We will give a description of the functors $\iota^*$, $\pi^*$ and $\pi^G_*$ for torus-equivariant reflexive sheaves in terms of families of filtrations. Making use of the equivalence of categories $\fK$, by abuse of notations, we will use the same symbols to design the associated functors on families of filtrations. 
\begin{lemma}
 \label{lem:restriction functor}
 Let $\cE$ be an equivariant reflexive sheaf on $X$. Assume that $\cE=\fK(\E)$, with $\E=\lbrace (E^F(i))\subset E: F\prec P, i\in \Z \rbrace$. Then the restriction $\cE^s:=\iota^*\cE$  is an equivariant reflexive sheaf on $X^s$ defined by the family of filtrations $\iota^*\E=\lbrace (E^F(i))\subset E: F\prec P^s, i\in\Z \rbrace$.
\end{lemma}
\begin{proof}
 The proof follows from the fact that if $Q\prec P^s$, then for every facet $F$ that contains $Q$, $F\prec P^s$. Indeed, if $F$ contains $Q$, it intersects $U$ and thus lies in $P^{ss}=P^s$. Thus by definition of $\iota^*\E$, and from equation (\ref{eq:sheaf from family of filtrations}), we obtain the result.
 \end{proof}
\begin{lemma}
 \label{lem:pushf functor}
 Let $\cE$ be an equivariant reflexive sheaf on $X$. Assume that $\cE=\fK(\E)$, with $\E=\lbrace (E^F(i))\subset E: F\prec P, i\in \Z \rbrace$. Then the invariant pushforward $\pi_*^G\cE^s$ is an equivariant reflexive sheaf on $Y$ defined by the family of filtrations $\pi_*^G\E^s = \lbrace (\check E^{\check F}(i))\subset \check E:\check F\prec P_Y, i\in\Z \rbrace$, where
\begin{itemize}
\item $\check{E}=E$,
\item $\check E^{\check F}(i)=E^F(b_F i)$ for all $F\prec P^s$, where we set $\check F= F\cap U$.
\end{itemize} 
\end{lemma}
 \begin{proof}
Let $F\prec P^s$. Then we have the GIT quotient projection 
 $$\pi : U_{\sigma_F}=\Spec(\C[\sigma_F^\vee\cap M]) \to  U_{\sigma_{F \cap U}}=\Spec(\C[\sigma_{U\cap F}^\vee\cap N_0^\perp])$$
 and by definition,
 $$
 \Gamma(U_{\sigma_{F \cap U}},\pi_*^G\cE^s)=\bigoplus_{m\in N_0^\perp}
 \Gamma(U_{\sigma_F},\cE^s)_m.
 $$
 Thus for all $m\in N_0^\perp$:
 $$
 \Gamma(U_{\sigma_{F \cap U}},\pi_*^G\cE^s)_m=
 \Gamma(U_{\sigma_F},\cE^s)_m
 $$
 that is
 $$
 \check E^{U\cap F}(\langle m,\check u_F\rangle)\otimes \chi^m = E^F(\langle m, u_F\rangle)\otimes \chi^m.
 $$
 Using Lemma \ref{lem:relations of m_F} and the basis $\check\cB_F^*$ and $\cB_F^*$ to decompose $m$ we obtain the result.
\end{proof}
Similarly, we have
\begin{lemma}
 \label{lem:pullb functors}
Let $\check{\cE} $ be an equivariant reflexive sheaf on $Y$. Assume that $\check\cE=\fK(\check\E)$, with $\check\E=\lbrace (\check E^{\check F}(i))\subset \check E: \check F\prec P_Y, i\in \Z \rbrace$. Then the pull-back   $\pi^*\check{\cE}$  is an equivariant reflexive sheaf on $X^s$ defined by the family of filtrations $\pi^*\check{\mathbb{E}}=\{ (E^F(i))\subset E:F\prec P^s,\, i\in\mathbb{Z}\} $ where
\begin{itemize}
\item $E=\check{E}$,
\item $E^F(i)=\check{E}^{F\cap U} (\lfloor\frac{i}{b_F}\rfloor)$ for all $F\prec P^s$.
\end{itemize} 
\end{lemma}
In particular, the dimension jump values $e^F(i)$ can only be non-zero for values of $i$ that are multiples of $b_F$. 

\begin{proof}
 By Lemma \ref{lem:pushf functor}, we know that
 $E^F(b_Fi)=\check E^{U\cap F}(i)$
 for $i\in\Z$, and hence we also have $E=\check E$.
 In general, we have for $j,i\in\Z$:
 $$
 \check E^{U\cap F}(j)\otimes \chi^{j \check m_F}=\Gamma(U_{\sigma_{F\cap U}}, \check \cE)_{j\check m_F}
 $$
 and
 $$
 E^{U\cap F}(i)\otimes \chi^{i m_F}=\Gamma(U_{\sigma_F}, \pi^*\check\cE)_{i m_F}.
 $$
 Moreover, we have by definition
 $$
 \Gamma(U_{\sigma_F},\pi^*\check \cE)=\Gamma(U_{\sigma_{F\cap U}}, \check \cE)\otimes_{\C[\sigma_{U\cap F}^\vee\cap N_0^\perp]} \C[ \sigma_F^\vee\cap M].
 $$
 Note that $m\in M\cap \sigma_F^\vee$ if and only if $\langle m,u_F \rangle \geq 0$ and  $m\in N_0^\perp\cap \sigma_{U\cap F}^\vee$ if and only if $\langle m,\check u_F\rangle  \geq 0$. Thus, to prove the lemma, it is enough to show that if $im_F=m'+(i m_F - m')$ with $m'\in N_0^\perp\cap \sigma_{U\cap F}^\vee$, and $(im_F -m')\in M\cap \sigma_F^\vee$ then $\langle m',\check u_F\rangle  \leq \lfloor\frac{i}{b_F}\rfloor$ (note that we use the fact that the spaces $(E^F(i))$ form a filtration). Suppose then that we have such a decomposition $im_F=m'+(i m_F - m')$ for $i m_F$. Without loss of generality, we can assume that $m'=a \check m_F$ with $a=\langle m',\check u_F\rangle \in \N$. Then we have $\langle  i m_F - a \check m_F,u_F\rangle  \geq 0$ and thus by Lemma \ref{lem:relations of m_F} we obtain $i- a b_F \geq 0$. The result follows.
\end{proof}
From these lemmas we deduce a version of Nevins' criterion for descent of equivariant reflexive sheaves on toric varieties \cite{Ne}:
\begin{corollary}
 \label{cor:Nevins reflexive}
 Let $\cE=\fK(\E)$ be an equivariant reflexive sheaf on $X$ with dimension jumps $(e^F(i))$.
 Then $\cE$ descends to $Y$ if and only if for all facets $F\prec P^s$, for all $i\in\Z$ such that $e^F(i)\neq 0$, $b_F$ divides $i$.
\end{corollary}
\begin{proof}
 The proof follows from the fact that $\cE$ descends if and only if $$\iota^*\E\simeq \pi^*  \pi^G_* \iota^*\E.$$ Using Lemmas \ref{lem:pushf functor} and \ref{lem:pullb functors}, this is equivalent that for all $F\prec P^s$, for all $i\in\Z$,
 $$
 E^F(i)= E^F\left(b_F \left\lfloor\frac{i}{b_F}\right\rfloor\right).
 $$
 The result follows.
\end{proof}
For later use, we emphasize the following corollary:
\begin{corollary}
 \label{cor:comparison of datas for slopes}
 Assume that $\cE$ is an equivariant reflexive sheaf on $X$ that descends to $Y$, with $\cE=\fK(\E)$. Then for all $W\subset E$, the subsheaf $\cE_W\subset \cE$ descends to $Y$.
\end{corollary}

\subsection{Pullback functors}
\label{sec:pullback functors}
We preserve notation of the previous section, still assuming the pair $(G,\gamma)$ to be generic. In this section, we introduce pull-back functors from $\Ref^T(Y)$ to $\Ref^T(X)$. Note that the superscript $T$ stands for equivariant, and means $T_X$-equivariant on $X$, and $T_Y$-equivariant on $Y$. The images of these functors will contain precisely the equivariant reflexive sheaves on $X$ that descend to $Y$ and that are suitable to compare slope stability notions on $X$ and $Y$.

Let $\bfD^u:=\lbrace D_F:\: F \prec P^{us} \rbrace$ be the set of unstable irreducible $T$-equivariant divisors on $X$.
For ${\bf i}:=(i_D)_{D\in \bfD^u}\in \Z^u$, we define a functor $\fP_{\bf{i}}$ from 
$\Ref^T(Y)$ to $\Ref^T(X)$. Using the functor $\fK$, it is sufficient to define the corresponding functor from $\Filt(P_Y)$ to $\Filt(P)$. Let $\check \E\in \Filt(P_Y)$, with $\check \E=\lbrace (\check E^{\check F}(i))\subset \check E: \check F\prec P_Y, i\in \Z \rbrace$. Then we define $\fP_\bfi(\check \E)=\lbrace (E^F(i))\subset E:F\prec P, i\in \Z \rbrace$ by setting $E =\check E$ and for $F\prec P$ and $i\in \Z$:
\begin{itemize}
 \item If $F\prec P^s$, then $E^F(i)=\check E^{U\cap F}(\lfloor\frac{i}{b_F}\rfloor)$. 
 \item If $F\prec P^{us}$, then 
\begin{equation*}
E^F(i)=
\left\{ 
\begin{array}{ccc}
0 & \mathrm{ if } & i < i_{D_F}\\
E & \mathrm{ if } & i \geq i_{D_F}
\end{array}
\right.
.
\end{equation*}
\end{itemize}
We now define the functors $\fP_\bfi$ on morphisms of families of filtrations. Let $\check \E_1, \check \E_2 \in \Filt(P_Y)$, with $\check \E_j=\lbrace (\check E_j^{\check F}(i))\subset \check E_j: \check F\prec P_Y, i\in \Z \rbrace$, for $j\in\lbrace 1,2 \rbrace$. Let $\check f\in \Hom(\check \E_1,\check \E_2)$, and recall that $\check f$ corresponds to a linear map $\check f: \check E_1 \to \check E_2$ such that for all $\check F\prec P_Y$ and $ i\in \Z$, $\check f(\check E_1^{\check F}(i))\subset \check E_2^{\check F}(i)$. We then simply define 
$\fP_\bfi(\check f):=f$ to be the element of $\Hom(\fP_\bfi(\check \E_1),\fP_\bfi(\check \E_2))$ given by the linear map $$ f=\check f : E_1=\check E_1 \to E_2=\check E_2.$$
We then check that $\fP_\bfi(\check f)$ is well defined, that is for all $F\prec P$ and $i\in \Z$, 
$$f(E_1^F(i))\subset E_2^F(i).$$
If $F\prec P^{us}$, this is clear as for $i < i_{D_F}$, $E_1^F(i)=E_2^F(i)=0$, while for $i\geq i_{D_F}$, $E_1^F=E_1$ and $E_2^F=E_2$. 
If $F\prec P^s$, then $E_j^F(i)=\check E_j^{U\cap F}(\lfloor\frac{i}{b_F}\rfloor)$, $j\in\lbrace 1,2 \rbrace$. However, by the compatibility of $\check f=f$ with the filtrations, for any $i'\in\Z$, we have 
$$
\check f(\check E_1^{U\cap F}(i'))\subset \check E_2^{U\cap F}(i'),
$$
which implies $f(E_1^F(i))\subset E_2^F(i)$. Thus $f$ is compatible with the families of filtrations and $\fP_\bfi(\check f)$ is well defined.

By abuse of notation, we denote by $\fP_\bfi$ the associated functors on $\Ref^T(Y)$. We then have:
\begin{proposition}
 \label{prop:fff functors}
 The functors $\fP_{\bf i}$ are fully faithful functors from the category $\Ref^T(Y)$ to the category $\Ref^T(X)$.
\end{proposition}
\begin{proof}
 We keep the notations used to define $\fP_\bfi$. We need to show that the map 
 $$
 \fP_\bfi : \Hom(\check \E_1,\check \E_2) \to \Hom(\fP_\bfi(\check \E_1),\fP_\bfi(\check \E_2))
 $$
 is a bijection. Injectivity is clear by definition of $\fP_\bfi(\check f)$. For surjectivity, let $f\in \Hom(\fP_\bfi(\check \E_1),\fP_\bfi(\check \E_2))$. Then $f$ is described by a linear map from $E_1$ to $E_2$ that is compatible with the families of filtrations $\fP_\bfi(\check \E_1)$ and $\fP_\bfi(\check \E_2)$. By definition of $\fP_\bfi$, for $F\prec P^{us}$, the condition $f(E_1^F(i))\subset E_2^F(i)$ is satisfied for any linear map from $E_1$ to $E_2$. However, for $F\prec P^s$, the compatibility gives for all $i\in\Z$:
 \begin{eqnarray*}
\displaystyle f\left(\check E_1^{U\cap F}\left(\left\lfloor\frac{i}{b_F}\right\rfloor\right)\right)\subset \check E_2^{U\cap F}\left(\left\lfloor\frac{i}{b_F}\right\rfloor\right).
 \end{eqnarray*}
 Taking in particular $i\in b_F\Z$, and using the fact that any facet of $P_Y$ is of the form $F\cap U$ for some facet $F\prec P^s$, we deduce that $f$ is a linear map between $\check E_1$ and $\check E_2$ that is compatible with the families of filtrations $\check \E_1$ and $\check \E_2$. Thus $f=\fP_\bfi(\check f)$ for $\check f = f$, seen as a linear map from $\check E_1=E_1$ to $\check E_2=E_2$.
\end{proof}
 \begin{remark}
 There is a simple relation between the different functors $\fP_\bfi$.
  Given an equivariant reflexive sheaf $\cE$ on $Y$, for any ${\bf i}=(i_D)_{D\in \bfD^u}\in \Z^u$, one has 
  $$
  \fP_\bfi(\cE)=\fP_{\bf 0}(\cE) \otimes \cO_X(-\sum_{D\in \bfD^u} i_D D).
  $$
  As will appear clearly in the next section, the difference of slopes between two reflexive sheaves is not affected by tensoring both sheaves by $\cO_X(-\sum_{D\in \bfD^u} i_D D)$. Thus, from the point of view of moduli spaces of slope stable sheaves, one is reduced to the study of the functor $\fP_{\bf 0}$. We thank the referee for this interesting remark.
 \end{remark}
 \begin{remark}
 The functors $\fP_\bfi$ might not preserve the compatibility conditions for local freeness introduced by Klyachko \cite{Kl}. We will give an explicit example of such phenomenon in Section \ref{sec:wpp quotients} (Example \ref{ex:p1p1p1top2}).
Note however that in the projective bundles examples (Section \ref{sec:proj bundles}), $\fP_0$ is just a pullback and thus preserves local freeness. 
 \end{remark}
We give a geometric interpretation of the images of the pull-back functors. First, we interpret the condition on unstable divisors with the following lemma, whose proof is straightforward from the definitions.
\begin{lemma}
 \label{lem:vanishing of Delta condition,geometric}
 Let $\mathcal{E}$ be an equivariant reflexive sheaf on $X$. Assume that $\cE=\fK(\E)$, with $\E=\lbrace (E^F(i))\subset E: F\prec P, i\in \Z \rbrace$. Let $F$ be a facet of $P$ corresponding to an irreducible invariant divisor $D_F$ and let $i_F\in \Z$. Then the following are equivalent:
 \begin{itemize}
  \item[i)] For all $i\in\Z$,
\begin{equation}
\label{eq:simple extension condition}
E^F(i)=
\left\{ 
\begin{array}{ccc}
0 & \mathrm{ if } & i < i_F\\
E & \mathrm{ if } & i \geq i_F
\end{array}
\right.
.
\end{equation}
  \item[ii)] For all $m\in M$, a non-zero element of $\Gamma(U_{\sigma_F}\setminus O(F) ,\cE)_m$ extends across the divisor $D_F$ to a section of $\Gamma(U_{\sigma_F},\cE)$ if and only if $\langle m, u_F\rangle \geq i_F$.
 \end{itemize}
\end{lemma}
\begin{remark}
 Condition $(ii)$ in Lemma \ref{lem:vanishing of Delta condition,geometric} is equivalent to the vanishing of the quantities $(\Delta_j(k))$ introduced in \cite[Proposition 3.20]{Koo11}, for the index $j$ corresponding to $F$.
\end{remark}
We deduce:
\begin{corollary}
 \label{cor:image of the functors}
 Let $\mathcal{E}$ be an equivariant reflexive sheaf on $X$, and let $\bfi\in \Z^u$. Then the following are equivalent:
 \begin{itemize}
  \item[i)] The sheaf $\cE$ lies in the image of $\fP_\bfi$.
  \item[ii)] The sheaf $\cE$ descends to $Y$ and satisfies the extension condition $(\ref{eq:simple extension condition})$ for all unstable divisor $D_F$, with $i_F=i_{D_F}$.
 \end{itemize}
\end{corollary}
\begin{proof}
 It follows from the definition of the functors and Lemmas \ref{lem:pullb functors} and \ref{lem:vanishing of Delta condition,geometric}
\end{proof}
It will become clear in the next section that the extension condition (\ref{eq:simple extension condition}) for unstable divisors is precisely the condition on $\fP_\bfi$ that enables us to compare the slopes of $\check \cE\in \Ref^T(Y)$ and of $\fP_\bfi(\check \cE)$ for all $\check \cE\in \Ref^T(Y)$.
We finish this section with a more effective lemma that will be used in these comparisons:
\begin{lemma}
 \label{lem:comparison jumps functors}
 Let $\check\cE\in\Ref^T(Y)$ and $\cE=\fP_\bfi(\check\cE)\in\Ref^T(X)$. Let $(e^F(i))$ be the dimension jumps of $\cE$ and let $(\check e^F(i))$ be the dimension jumps of $\check \cE$. Then
 \begin{enumerate}
  \item[i)] For all $F\prec P^s$, $e^F(i)=0$ if $i\neq 0 \mod b_F$ and $e^F(b_F i)=\check e^F(i)$.
  \item[ii)] For all $F\prec P^{us}$, $e^F(i)=0$ if $i\neq i_{D_F}$ and $e^F(i_{D_F})=\rk(\cE)$.
 \end{enumerate}
\end{lemma}
\begin{proof}
 It follows directly from the definition of the functors.
\end{proof}

 
\section{Slopes under descent and the Minkowski condition}
\label{sec:slopes and minkowski}
In this section, $(X,L)$ is a polarized toric variety with torus $T=N\otimes_\Z \C^*$, and $G=N_0\otimes_\Z \C^*$ is a subtorus. We assume $N_0=N\cap (N_0\otimes\R)$ and consider a linearization $\gamma$ such that $(G,\gamma)$ is generic (as in Definition \ref{def:generic Cstar}). We will compare the slope stability notions on $X$ and $Y=X/\!\!/G$.
We will use the same notation as in  previous sections.
 \subsection{Some stability notions} 
 \label{sec:stability notions}
We recall now the notions of stabilities that we will consider, and state some propositions specific to the toric setting. We refer to \cite{HuLe,Gue} for the stability notions (see in particular \cite[Proposition 4.3]{Gue} for the definition of slope on normal varieties) and to \cite{La} for the definition of the intersection of a Weil divisor with real Cartier divisors.
\begin{definition} Let $\mathcal{E}$ be a torsion-free coherent sheaf on $X$. The \emph{degree} of $\mathcal{E}$ with respect to an ample class $\alpha\in N^1(X)_\R$ is the real number obtained by intersection:
\begin{eqnarray*}
\deg_\alpha(\mathcal{E})=c_1(\mathcal{E})\cdot \alpha^{n-1}
\end{eqnarray*}
and its \emph{slope} with respect to $\alpha$ is given by
\begin{eqnarray*}
\mu_\alpha(\mathcal{E}) = \frac{\deg_\alpha(\mathcal{E})}{\rk(\mathcal{E})}.
\end{eqnarray*}
A torsion-free coherent sheaf $\mathcal{E} $ is said to be $\mu$-\emph{semi-stable} or \emph{slope semi-stable} with respect to $\alpha$ if for any coherent subsheaf $0\neq \mathcal{F}\subsetneq\mathcal{E}$ with $0<\rk\, \cF<\rk\,\cE$, one has 
\begin{eqnarray}
\label{eq:slopeinequality}
\mu_\alpha(\mathcal{F})\leq \mu_\alpha(\mathcal{E}).
\end{eqnarray}
When strict inequality always holds, we say that $\mathcal{E}$ is $\mu$-\emph{stable}. Finally, $\cE$ is said to be $\mu$-{polystable} if it is the direct sum of $\mu$-stable subsheaves of the same slope.
\end{definition}
It is a standard fact that to test stability of a torsion-free coherent sheaf $\cE$, it is enough to check (\ref{eq:slopeinequality}) for proper saturated coherent subsheaves $\cF$ (see \cite[Proposition 1.2.6]{HuLe}). A less standard fact that follows from \cite[Proposition 4.13]{Koo11} is that in the toric setting, to check $\mu$-stability of an equivariant reflexive sheaf, it is enough to compare slopes for saturated equivariant reflexive subsheaves. More precisely, using the description of saturated equivariant reflexive subsheaves from Proposition \ref{prop:equiv ref subsheaves}, we have:
\begin{proposition}{\rm \cite[Proposition 4.13]{Koo11} and \cite[Proposition 2.3]{HNS}. }
\label{prop:stability equiv subsheaves}
Let $\cE=\fK(\E)$ be a $T$-equivariant reflexive sheaf on a normal toric variety $X$. Then $\cE$ is $\mu$-semi-stable (resp. $\mu$-stable) if and only if for all proper vector subspaces $W\subset E$, $\mu_L(\cE_W)\leq \mu_L(\cE)$ (resp. $\mu_L(\cE_W)< \mu_L(\cE)$), where $\cE_W=\fK(W\cap \E)$.
\end{proposition}
Note that the proof of the above proposition is valid on normal toric varieties. Following \cite{Koo11,HNS}, we will use the following combinatorial formula for the slopes of equivariant reflexive sheaves (see also \cite[Lemma 2.2]{HNS}):
\begin{lemma}
 \label{lem:formula slope toric}
  Let $\cE$ be a $T$-equivariant reflexive sheaf with dimension jumps $(e^F(i))$. Then
 $$
 \mu_L(\cE)=-\frac{1}{\rk(\cE)}\sum_{F\prec P} i_F(\det\cE) \deg_L(D_F),
 $$
 where for all $F\prec P$,
 $$
 i_F(\det\cE)=\sum_{i\in\Z} i e^F(i).
 $$
\end{lemma}
\begin{proof}
 It follows from the definition of slopes and Proposition \ref{prop:first Chern class}.
\end{proof}
We will need to compare the data defining the slopes after descent:
\begin{lemma}
 \label{lem:comparison slopes}
 Let $\bfi=(i_F)_{F\prec P^{us}}\in\Z^u$.
 Let $\check\cE$ be a $T$-equivariant sheaf on $Y$. Then:
 \begin{itemize}
  \item[i)] For all $F\prec P^s$,
 $ i_F(\det\fP_\bfi(\check \cE))=b_F i_{F\cap U}(\det\check\cE).$
  \item[ii)] For all $F\prec P^{us}$, $i_F(\det\fP_\bfi(\check \cE))=i_F\: \rk(\cE)$.
 \end{itemize}
\end{lemma}
\begin{proof}
 The proof follows from the definition of $i_F(\det\cE)$ and Lemma \ref{lem:comparison jumps functors}.
\end{proof}
We will also need the degree for the pullbacks of the irreducible $T$-invariant divisors.
\begin{lemma}
 \label{lem:comparison line bundes}
 Let $F_1\prec P^s$. Then for any $\bfi\in\Z^u$,
 $$
 \mathfrak{P}_\bfi(\mathcal{O}_Y(D_{F_1\cap U}))=\mathcal{O}_X(b_{F_1}D_{F_1}-\sum_{F\prec P^{us}}i_FD_F).
 $$
\end{lemma}
\begin{proof}
 The sheaf $\mathcal{O}_Y(D_{F_1\cap U})$ is equivariant and reflexive. Combining Proposition \ref{prop:rank one} and the definition of the functor $\mathfrak{P}_\bfi$, $\mathfrak{P}_\bfi(\mathcal{O}_Y(D_{F_1\cap U}))$ is determined by the family of filtrations    
 \begin{eqnarray*} 
 L^{F_1}(i)&=& 
 \begin{cases}
 0 & i<-b_{F_1},\\
 \mathbb{C} & i\geq -b_{F_1},
 \end{cases}
  \\ 
L^F(i)&=& 
 \begin{cases}
 \begin{array}{cc}
 0 & i<0,\\
 \mathbb{C} & i\geq 0,
 \end{array}
 \text{ if } F\neq F_1 \text{ is stable},\\
   \begin{array}{cc}
 0 & i<i_F,\\
 \mathbb{C} & i\geq i_F,
 \end{array}
 \text{ if } F \text{ is unstable}.  
\end{cases} 
  \end{eqnarray*}   
 That is to say, $\mathfrak{P}_\bfi(\mathcal{O}_Y(D_{F_1\cap U}))=\mathcal{O}_X(b_{F_1}D_{F_1}-\sum_{F\prec P^{us}}i_FD_F)$.
 \end{proof}
  \begin{corollary}
   \label{cor:comparison degree divisors}
   Let $F_1\prec P^s$. Then for any $\bfi\in\Z^u$,
 $$
\deg_L( \mathfrak{P}_\bfi(\mathcal{O}_Y(D_{F_1\cap U})))=b_{F_1} \deg_L(D_{F_1})-\sum_{F\prec P^{us}}i_F \deg_L(D_F).
 $$
  \end{corollary}

\subsection{Comparison of slope stability via the Minkowski condition}
\label{sec:Minkowski condition} 
In this section, we prove Theorem \ref{theo:Minkowski existence} and Proposition \ref{prop:mink converse}, from which Theorem \ref{theo:intro} follows. Recall from Sections \ref{sec:toricpolarized} and \ref{sec:toricGIT} that the data of $(X,L)$ together with the linearization $\gamma : T \to \Aut(L)$ determines a polytope $P \subset M_\R$, and that facets $F$ in the stable locus $P^s$ of $P$ correspond to stable irreducible $T$-invariant  divisors $D_F$.


\begin{theorem}\label{theo:Minkowski existence}
Let $(X,L)$ be a polarized toric variety with torus $T=N\otimes_\Z\C^*$. Let $G=N_0\otimes_\Z\C^*$ be a subtorus for a saturated sublattice $N_0\subset N$. Let $\gamma : T\to \Aut(L)$ be a generic linearization of $G$, and denote by $Y$ the associated GIT quotient $X/\!\!/G$. 
Assume that Minkowski condition holds:
\begin{eqnarray}
\label{eq:Minkowski condition}
\sum_{F\prec P^s} \deg_L(D_F)\: u_F=0  \mod N_0\otimes_\mathbb{Z}\mathbb{R}.
\end{eqnarray}
Then there exists a unique ample class $\alpha\in N^1(Y)_\R$ such that for every equivariant reflexive sheaf $\check{\mathcal{E}}$ on $Y$, for any $\bfi=(i_F)_{F\prec P^{us}}\in\Z^u$, setting $\mathcal{E}=\mathfrak{P}_\bfi(\check{\mathcal{E}})$, we have   
\begin{eqnarray}
\label{eq:equality slopes theo}
\mu_L(\mathcal{E})=\mu_{\alpha}(\check{\mathcal{E}})-\sum_{F\prec P^{us}}i_F\mathrm{deg}_L(D_F).
\end{eqnarray}
In this case, $\mathcal{E}$ is stable on $(X,L)$ if and only if $\check{\mathcal{E}}$ is stable on $(Y,\alpha)$.
\end{theorem}
A converse of this statement is given in the following. 
\begin{proposition}\label{prop:mink converse}
There exists an ample class $\alpha\in N^1(Y)_\mathbb{R}$ with respect to which the functors $\mathfrak{P}_\bfi$ from equivariant reflexive sheaves on $Y$ to equivariant reflexive sheaves on $X$ preserve each of the conditions of $\mu$-stability, $\mu$-semi-stability and $\mu$-polystability if and only if 
\begin{eqnarray*}  
\sum_{F \prec P^s}\deg_P(D_F)u_F=0 \mod\ N_0\otimes_\mathbb{Z}\mathbb{R}. 
\end{eqnarray*}
\end{proposition} 

We delay the proofs slightly in order to first present the classical Minkowski theorem. We recall that a lattice $M $ defines a measure $\nu$ on $M_\mathbb{R}=M\otimes_\mathbb{Z}\mathbb{R}$ as the pull-back of Haar measure on $M_\mathbb{R}/M$. It is determined by the properties 
\begin{enumerate} 
\item[(a)] $\nu$ is translation invariant, 
    \item[(b)] $\nu(M_\mathbb{R}/M)=1$.
\end{enumerate}
Let $\alpha\in N^1(X)_\R$ be an ample class, and denote by $P_\alpha$ the associated polytope.
For any facet $F$ of $P_\alpha$, there is a unique $(n-1)$-form $\nu_F$, independent of $\alpha$, such that $\nu_F\wedge u_F = \nu$. If $P_\alpha$ is a lattice polytope, with
$$
P_\alpha=\lbrace m\in M_\R : \langle m , u_F\rangle\geq - a_F \textrm{ for all } F\prec P \rbrace,
$$
the form $\nu_F$ corresponds to the measure defined by the lattice 
$$
M\cap \lbrace m\in M_\R : \langle m , u_F\rangle= - a_F \rbrace
$$ in the affine span of $F$. We will denote the volume of the facet $F$ with respect to $\nu_F$ by ${\rm latvol}(F)$. 
\begin{proposition}[\cite{Dan}]
 \label{prop:lattice volume is degree}
 For any facet $F$ of $P_\alpha$,
 $$
 \deg_\alpha(D_F)=\mathrm{latvol}(F).
 $$
\end{proposition}
\begin{proof}
From \cite[Section 11]{Dan}, the equality holds for $\alpha$ integral.
Since the two expressions in the equality are both homogeneous (of order $n-1$) and continuous, we also have  $\deg_\alpha(D_F)=\mathrm{latvol}(F)$, whenever $\alpha$ is an ample $\mathbb{R}$-divisor.
\end{proof} We have the elementary lemma:
\begin{lemma}
Let $u_F\in N$ be the primitive element associated to a facet $F\prec P$. Let $g=\langle\cdot,\cdot\rangle$ be a positive definite inner product on $M_\mathbb{R}$ whose induced volume form coincides with $\nu$ and let $\tilde{u}_F= u_F/\|u_F\|$ be the normalized vector in $N_\mathbb{R}$. Then, 
\begin{eqnarray*}
{\rm latvol}(F)\,u_F={\rm vol}_g(F)\,\tilde{u}_F.
\end{eqnarray*}
\end{lemma} Here ${\rm vol}_g$ denotes the $(n-1)$-dimensional Euclidean volume of the facet $F$, calculated with respect to the inner product $g$.
From this observation we obtain the counterpart of a classical result of Minkowski, adapted to the case of lattice polytopes.
\begin{proposition}
\label{prop:lattice Minoswski}
Let $u_1,\ldots,u_r\in N$ be distinct primitive lattice elements that span the real vector space $N_\mathbb{R}$ and let $f_1,\ldots, f_r>0$ be positive real numbers. Then there exists a compact and convex polytope in $M_\R$ whose facets have inward normals the elements $(u_i)$ and lattice volumes the $(f_i)$ if and only if 
\begin{eqnarray*}
\sum_{i=1}^rf_iu_i=0.
\end{eqnarray*}
Moreover, such a polytope is unique up to translation.
\end{proposition}
\begin{proof}
 This follows immediately from the classical result (see \cite[p.455]{Sch}) and the previous lemma.
\end{proof}
We deduce the following interesting corollary on intersection theory of toric varieties:
\begin{corollary}
 \label{cor:intersection theoretical Minkowski}
 Let $(\alpha_\rho)_{\rho\in \Sigma(1)}\in \R_{> 0}^{\sharp\Sigma(1)}$. For any ample class $\alpha\in N^1(X)_\R$ denote by $P_\alpha$ the associated polytope. Then the following two statements are equivalent: 
 \begin{itemize}
  \item[i)] There exists an ample class $\alpha\in N^1(X)_\R$ such that for all facets $F\prec P_\alpha$ , $\deg_\alpha(D_F)=\alpha_{\rho_F}$,
  \item[ii)] The following condition is satisfied:
 $$
 \sum_{F\prec P} \alpha_{\rho_F}\: u_F = 0.
 $$
 \end{itemize}

\end{corollary}

 We can now give the proofs of Theorem \ref{theo:Minkowski existence} and Proposition \ref{prop:mink converse}.
\begin{proof}[Proof of Theorem \ref{theo:Minkowski existence}]
Let $\bfi\in \Z^u$.
Let $\check{\mathcal{E}}$ be an equivariant reflexive sheaf on $Y$ and let $\mathcal{E}=\mathfrak{P}_\bfi(\check{\mathcal{E}})$ on $X$, both of rank $r$. Let $\alpha$ be an ample $\mathbb{R}$-divisor on $Y$ that determines a real polytope $\check P\subseteq \check{M}_\mathbb{R}= N_0^\perp\otimes_\mathbb{Z}\mathbb{R}$ whose facets $\check{F}$ have primitive normals $\check u_F=u_{F\cap U}\in N/N_0$ equal to those of $P_Y$.
Then, from Lemma \ref{lem:formula slope toric},
\begin{eqnarray*}
\mu_\alpha(\check{\cE})=-\frac{1}{r}\sum_{ \check F\prec \check P}   i_{\check F}(\det \check{\mathcal{E}})\,\mathrm{deg}_{\alpha}(D_{\check F}).
\end{eqnarray*}
By the genericity assumption on $(G,\gamma)$, from Lemma \ref{lem:genericity lemma}, we deduce a bijective correspondence $F \leftrightarrow \check F$ between facets of $P^s$ and facets of $\check P$. Thus
\begin{eqnarray*}
 \mu_\alpha(\check{\cE})= -\frac{1}{r}\sum_{ F\prec P^s}   i_{\check F}(\det \check{\mathcal{E}})\,\mathrm{deg}_{\alpha}(D_{\check F}),
\end{eqnarray*}
whereas by Lemma \ref{lem:comparison slopes},
\begin{eqnarray*}
\mu_L(\mathfrak{P}_\bfi(\check{\mathcal{E}})) &=& \mu_L(\mathcal{E}),\\
&=& -\frac{1}{r}\sum_{F\prec P}i_F(\det \mathcal{E})\,\mathrm{deg}_L(D_F),\\
&=& -\frac{1}{r}\sum_{F\prec P^s} i_{\check F}(\det \check{\mathcal{E}})b_F\mathrm{deg}_L(D_F) - \sum_{F\prec P^{us}}i_F\mathrm{deg}_L(D_F).
\end{eqnarray*} 
Consider the right hand term of the last equality. We note two points:
 \begin{enumerate}
 \item The second term $\sum_{F\prec P^{us}}i_F\deg_L(D_F) $ is independent of the reflexive sheaf $\check{\mathcal{E}}$. It depends only on the lattice volumes of the unstable facets $F\prec P^{us}$ and on the indices $i_F$ that determine the pull-back functors.
 \item The first term coincides with $\mu_\alpha(\check\cE)$ if for all $F\prec P^s$, $\deg_\alpha(D_{\check F})= b_F\deg_L(D_F)$, which, by Proposition \ref{prop:lattice volume is degree}, is equivalent to $\mathrm{latvol(\check F)}=b_F\deg_L(D_F)$. 
\end{enumerate} 
By Proposition \ref{prop:lattice Minoswski}, we can chose the polytope $\check{P}\subseteq \check{M}_\mathbb{R}$
such that its facets satisfy $\mathrm{latvol}(\check{F})=b_F\deg_L(D_F)$ if and only if 
\begin{eqnarray*}
\sum_{\check F\prec \check  P}b_F\mathrm{deg}_L(D_F)\check u_{ F}=0,
\end{eqnarray*}
which, by Lemma \ref{lem:projection primitive defining facet}, is to say
\begin{eqnarray*}
\sum_{F\prec P^s}\mathrm{deg}_L(D_F)u_F=0 \mod N_0\otimes_\Z \R. 
 \end{eqnarray*}
 Thus, the Minkowski condition implies the existence of the desired polytope and of $\alpha$.
 
 Applying now (\ref{eq:equality slopes theo}) to the sheaf $\cO(D_{\check F})$, using Corollary \ref{cor:comparison degree divisors}, we deduce that $\deg_\alpha D_{\check F}=b_F \deg_L(D_F)$ for all $\check F \subset \check P$. Then, the uniqueness statement of Theorem \ref{theo:Minkowski existence} follows from the equality $\mathrm{latvol}(\check F)=\deg_\alpha(D_{\check F})$ and unicity in Proposition \ref{prop:lattice Minoswski}.
 
As noted above in Proposition \ref{prop:stability equiv subsheaves} (see \cite{DDK,HNS}), to test for stability, it is sufficient to consider equivariant reflexive subsheaves of the form $\check{\mathcal{E}}_W=\mathfrak{K}(\mathbb{E}\cap W)$ where $W$ runs through vector subspaces of $E$. The slopes then satisfy
\begin{eqnarray*}
\mu_L(\mathcal{E})-\mu_L(\mathcal{E}_W)=\mu_\alpha(\check{\mathcal{E}})-\mu_\alpha(\check{\mathcal{E}}_W)
\end{eqnarray*}
from which the equivalence of the stability conditions follows.    
\end{proof}

\begin{proof}[Proof of Proposition \ref{prop:mink converse}  ]The ``if'' statement is given in the previous theorem. We suppose that for some ample $\mathbb{R}$-divisor $\alpha$ on $Y$ the various stability notions are preserved by the pull-back functors. The ample divisor $\alpha$ defines a real polytope $\check P\subseteq \check{M}_\mathbb{R}= N_0^\perp\otimes_\mathbb{Z}\mathbb{R}$ whose facets $\check{F}$ have primitive normals $\check u_{F}=u_{F\cap U}\in N/N_0$.

We consider the functor $\mathfrak{P}_\bfi$ with $i_F=0$ for all unstable facets $F\prec P^{us}$. For any two $F_1,F_2\prec P^s$ consider the direct sum of rank one reflexive sheaves $\mathcal{O}_Y(d_1 D_{\check F_1})\oplus \mathcal{O}_Y(d_2 D_{\check F_2})$ where $d_1=b_{F_2}\deg_L(D_{F_2})$ and $d_2=b_{F_1}\deg_L(D_{F_1})$. Then, using Lemma \ref{lem:comparison line bundes},
\begin{eqnarray*}
\mathfrak{P}_\bfi(\mathcal{O}_Y(d_1 D_{\check F_1})\oplus \mathcal{O}_Y(d_2 D_{\check F_2})) =\mathcal{O}_X(d_1 b_{F_1} D_{F_1 })\oplus \mathcal{O}_X(d_2 b_{F_2}D_{F_2 })
\end{eqnarray*}
is the direct sum of rank-one reflexive sheaves of the same slope. By hypothesis, the initial sheaf must also be $\mu$-polystable so for any two stable facets of $P$, 
\begin{eqnarray*}
\deg_\alpha(b_{F_2}\deg_L(D_{F_2})D_{\check F_1}) =  \deg_\alpha(b_{F_1}\deg_L(D_{F_1})D_{\check F_2})
\end{eqnarray*}
so the expression 
\begin{eqnarray*}
c=b_F\frac{\deg_L(D_F)}{\deg_\alpha(D_{\check F})}
\end{eqnarray*}
is independent of $F\prec P^s$. Then, $\mathrm{latvol}(\check{F})=\deg_\alpha(D_{\check{F}})$ and by the lattice Minkowski theorem (Proposition \ref{prop:lattice Minoswski}),
\begin{eqnarray*}
\sum_{F\prec P^s} \deg_\alpha(D_{\check F}) \check u_{F}=0.
\end{eqnarray*}
That is, by Lemma \ref{lem:projection primitive defining facet},
\begin{eqnarray*}
\frac{1}{c}\sum_{F\prec P^s} b_F\,\deg_L(D_F)\,b_F^{-1}\pi(u_F)=0
\end{eqnarray*}
as desired.
 \end{proof}

\section{Applications}
\label{sec:applications}
In this section, we investigate the Minkowski condition. We show that any $n$-dimensional polarized toric orbifold $(X,L)$ admits at least $n+1$ GIT quotients by $\C^*$-actions that satisfy the Minkowski condition. The associated quotients are weighted projective spaces. We also consider the case of projectivization of torus invariant bundles.

\subsection{Compatible one parameter subgroups}
\label{sec:compatible actions}
Let $(X,L)$ be a polarized projective toric variety of complex dimension $n$ with torus $T=N\otimes_\Z \C^*$ and fan $\Sigma$. When considering subtori $G\subset T$, we will always assume that $G=N_0\otimes_\Z \C^*$, for a saturated sublattice $N_0$.
\begin{definition}
 \label{def:compatible action}
  A subtorus $G\subset T$ is \emph{compatible} with $(X,L)$ if there is a linearization $\gamma$ of $T$ on $L$ such that $(G,\gamma)$ is generic and $((X,L),(G,\gamma))$ satisfies Minkowski condition (\ref{eq:Minkowski condition}).
\end{definition}
We will see that the Minkowski condition is very restrictive, in the sense that in most cases, for fixed $(X,L)$, there are few compatible one dimensional subtori. To see this, set $\mathrm{Div}_{\mathrm{ir}}^T(X)= \lbrace D_\rho, \: \rho \in \Sigma{(1)} \rbrace$ be the set of reduced and irreducible $T$-invariant divisors on $X$. Let $\bfD\subset \mathrm{Div}_{\mathrm{ir}}^T(X)$ be a non empty subset. Assume that there is a one parameter subgroup $G=N_0\otimes_\Z \C^*\subset T$ and a linearization $\gamma$ for $T$ on $L$ such that $(G,\gamma)$ is generic and the set of stable divisors for $(G,\gamma)$ is $\bfD$. Then, if Minkowski condition is satisfied, we have
$$
\sum_{D_\rho \in \bfD} \deg_L(D_\rho)\: u_\rho = 0 \mod N_0\otimes_\Z \R
$$
Consider the element of $N$:
\begin{equation*}
 \label{eq:vector family divisors}
 u_{\bfD}:= \sum_{D_\rho\in \bfD} \deg_L(D_\rho)\: u_\rho.
\end{equation*}
Assume that $u_{\bfD}$ is not zero. Then Minkowski condition implies that 
$$N_0=N\cap(\R\cdot u_\bfD),$$
so that in particular $G$ is entirely determined by $u_\bfD$ and thus by the set $\bfD$. So for a given $\bfD$ with $u_\bfD$ non-zero, there is at most one compatible one-parameter subgroup of $T$ with $\bfD$ as set of stable divisors. As the number of stable divisors is the number of faces of a given polytope in $M_\R$ that intersects the hyperplane $N_0^\perp\otimes_\Z \R$, in the generic situation, we must have $\sharp \bfD\in\lbrace n, \ldots, \sharp \Sigma(1)-1 \rbrace$. Thus we have proved:
\begin{lemma}
\label{lem:restrictive action}
 Let $(X,L)$ be a $n$-dimensional projective polarized toric variety with fan $\Sigma$. Set $d=\sharp \Sigma(1)$ the number of rays of $\Sigma$. Assume that for all subset $\bfD \varsubsetneq \mathrm{Div}_{\mathrm{ir}}^T(X)$ with $\sharp \bfD\in\lbrace n, \ldots, d-1 \rbrace$, we have $u_\bfD\neq 0$. Then,
 the number of one parameter subgroups compatible with $(X,L)$ is bounded by $\sum_{k=n}^{d - 1} \binom{d}{k}$.
\end{lemma}
For example, the conditions of Lemma \ref{lem:restrictive action} are satisfied by Hirzebruch surfaces $\F_a$ for $a\geq 2$, but not by $\P^1\times \P^1$.
\begin{remark}
\label{rem:picard rank}
 Note that if $X$ is a projective toric orbifold, then from the sequence
 $$
 0 \to M \to \bigoplus_{\rho\in \Sigma(1)} \Z\cdot D_\rho \to \mathrm{Cl}(X) \to 0
 $$
 and from the fact that the Picard group as finite index in the class group $\mathrm{Cl}(X)$, we deduce that $\sharp \Sigma(1)= n + p$, where $p$ is the Picard rank of $X$ (see \cite[Theorem 4.1.3 and Proposition 4.2.7]{CLS}). This gives the bound $\sum_{k=n}^{n+p - 1} \binom{n+p}{k}$ on the number of compatible one-parameter subgroups. In particular, for the complex projective space $\C\P^n$, this bound equals $n+1$. Using Proposition \ref{prop:orbifold to WPP}, we see that for a given dimension, this bound is achieved.
\end{remark}

\subsection{Quotients to weighted projective spaces}
\label{sec:wpp quotients}
In this section we will show:
\begin{proposition}
 \label{prop:orbifold to WPP}
 Let $(X,L)$ be a $n$-dimensional polarized toric orbifold with torus $T$ and associated polytope $P$ and denote by $m$ the number of torus fixed points of $X$. Then, up to replacing $L$ by a sufficiently high power,
 there exist at least $m$ one-parameter subgroups of $T$ compatible with $(X,L)$ giving distinct GIT quotients. The associated GIT quotients are weighted projective spaces.
\end{proposition}
From this proposition, we deduce that one can obtain $\mu$-stable reflexive sheaves on any polarized toric orbifold from $\mu$-stable reflexive sheaves on weighted projective spaces. 
\begin{remark}
As weighted projective spaces have Picard rank $1$, the number of rays of their fans, and thus the number of facets of their polytopes, is the smallest possible for a given dimension (see Remark \ref{rem:picard rank}). Thus, testing stability for reflexive sheaves is simpler on weighted projective spaces. We expect that these varieties could serve as simple bricks to study moduli spaces of equivariant reflexive sheaves on toric orbifolds, and will investigate in this direction in future work.
\end{remark}
\begin{proof}
 The argument is local at a vertex $v\in P$. As $X$ is an orbifold, its fan is simplicial, and thus the set $\lbrace u_F,\: v\in F\rbrace$  is a $\Q$-basis of $M_\Q$. Consider $$\bfD_v:=\lbrace D_F:\: v\in F \rbrace$$ and
 $$
 u_{\bfD_v}=\sum_{D_F\in \bfD_v} \deg_L(D_F) u_F=\sum_{v\in F} \deg_L(D_F) u_F.
 $$
 As all the degrees $\deg_L(D_F)$ are positive, it is clear that $u_{\bfD_v}$ is not zero. Let $N_0=N\cap (\R\cdot u_{\bfD_v})$ and let $G=N_0\otimes_\Z \C^*$. Then, by construction of $u_{\bfD_v}$, up to dilation and translations of $P$, we can assure that $P$ intersects $N_0^\perp$ transversally and precisely along the faces $F$ that contains $v$. Thus, up to scaling $L$ and choosing an appropriate linearization, the set of stable divisors for the $G$ action is $\bfD_v$. By construction, Minkowski condition is satisfied and $G$ is compatible with $(X,L)$. The associated quotient is a weighted projective space as the number of stable facets is $n$, giving $n$ rays for the $(n-1)$-dimensional toric GIT quotient.
 To conclude, the vertices in the polytope correspond to $0$-dimensional orbits, hence to fixed points. Thus, there are at least $m$ such quotients, hence the result.
\end{proof}

\begin{remark}
 \label{rem:CPn}
 From Lemma \ref{lem:restrictive action} and Proposition \ref{prop:orbifold to WPP}, we deduce that the compatible one parameter subgroups for $(\C\P^n,\cO_{\P^n}(1))$ are given by the $\C^*$-actions 
 $$\lambda \cdot [z_0,\ldots, z_n]=[z_0,\ldots, \lambda \, z_i, \ldots, z_n],\: i\in \lbrace 0, \ldots , n\rbrace,$$
 and the associated quotients are all isomorphic to $\C\P^{n-1}$.
\end{remark}

\begin{example}
\label{ex:p1p1p1top2}
 Let $X=\C\P^1\times\C\P^1\times\C\P^1$ together with the polarization $L=\pi_1^*\cO_{\P^1}(4)\otimes\pi_2^*\cO_{\P^1}(4)\otimes\pi_3^*\cO_{\P^1}(4)$ where $\pi_j: X\to \C\P^1$ denotes the projection onto the $j$-th factor. The associated polytope $P$ in $\Z^3\otimes_\Z\R=\R^3$ is a cube with edges of length $4$. Up to translation, that is up to changing the linearisation of the torus action on $L$, we can assume that $P$ is 
 $$
 P=[-3,1]\times[-3,1]\times[-3,1]\subset\R^3.
 $$
 Consider now the one parameter subgroup $G$ of $T$ whose lattice is $N_0=\Z\cdot (1,1,1)$. The GIT quotient of $(X,L)$ by $G$ is the toric variety $Y$ associated to the polytope $U\cap P=(1,1,1)^\perp\cap P$, that is
 $$
 P_Y=\lbrace (x,y,z)\in\R^3\,:\,\: x+y+z=0,\: x,y,z\in[-3,1]\rbrace.
 $$
 We recognize the polytope of $\C\P^2$, so that $Y\simeq \C\P^2$. The stable facets in $P^s$ are given by $F_x=\lbrace x=1 \rbrace \cap P$, $F_y=\lbrace y=1 \rbrace \cap P$ and $F_z=\lbrace z=1 \rbrace \cap P$. The degree of the associated invariant divisors $D_x, D_y$ and $D_z$ is equal to $16$ with respect to  $L$, so the Minkowski condition becomes
 $$
 16 u_{D_x} + 16 u_{D_y} + 16 u_{D_z}\in N_0\otimes_\Z\R=\R\cdot(1,1,1).
 $$
 As $u_{D_x}=(-1,0,0)$, $u_{D_y}=(0,-1,0)$ and $u_{D_z}=(0,0,-1)$, the Minkowski condition is satisfied. By Theorem \ref{theo:intro}, and from the fact that, up to scaling, the only polarization on $\C\P^2$ is $\cO_{\P^2}(1)$, an equivariant reflexive sheaf $\check \cE$ on $\C\P^2$ is slope stable with respect to $\cO_{\P^2}(1)$ if and only if $\fP_\bfi(\check\cE)$ is slope stable on $X$ with respect to $L$. For example, the tangent sheaf $\cT_{\C\P^2}$ is slope stable with respect to $\cO_{\P^2}(1)$ (see e.g. \cite[Theorem 1.3]{HNS}). Thus we obtain a slope stable rank $2$ equivariant reflexive sheaf on $(X,L)$ by setting $\cE=\fP_{\bf 0}(\cT_{\C\P^2})$. Note that this equivariant sheaf is not locally free, and hence $\fP_{\bf 0}$ doesn't preserve local freeness in general. To see this, recall from \cite{Kl} that the family of filtrations defining $\cT_{\C\P^2}$ is given by $\lbrace (T^{\check F}(i))\subset  N_Y\otimes_\Z \C: \check F\prec P_Y, i\in \Z\rbrace$ where
 $$ 
 T^{\check F}(i)=\left\{\begin{array}{ccc}
                0 & \: \mathrm{ if } \: & i\leq -2\\
                \mathrm{Vect}(\check u_F)& \: \mathrm{ if } \: & i=-1\\
              N_Y\otimes_\Z \C & \: \mathrm{ if } \: & i\geq 0.
               \end{array}
               \right.
 $$
 We compute $b_{F_x}=b_{F_y}=b_{F_z}=1$ here, so that the family of filtrations associated to $\cE$ is $\E=\lbrace ( E^{F}(i))\subset  E: F\prec P, i\in \Z \rbrace$, where in particular $E^{F}(i)=T^{F\cap U}(i)$ for $F\in\lbrace F_x, F_y, F_z \rbrace$. Klyachko's criterion for $\cE$ to be locally free imposes a condition on the family of filtrations $\E$ for each face of $P$ (see e.g. \cite[Theorem 5.22]{perl} for the statement with increasing filtrations). In particular, for the vertex $(1,1,1)\in P$, this condition states that there should be a $\Z^3$-grading of $E=N_Y\otimes_\Z \C$:
 $$
 E=\bigoplus_{m\in \Z^3} E_m
 $$
 such that for each face $F\prec P$ 
 containing $(1,1,1)$, one has
 $$
 E^F(i)=\sum_{\langle m,u_F\rangle\leq i} E_m.
 $$
 As for such faces, $\check u_F=\pi(u_F)$, where $\pi: \Z^3 \to N_Y=\Z^3/\Z\cdot(1,1,1)$ denotes the projection map, and $E^F(i)=T^{U\cap F}(i)$, we deduce that we would have 
 \begin{eqnarray*}
 E_{(1,0,0)}=\mathrm{Vect}(\check u_{F_x}),\: E_{(0,1,0)}=\mathrm{Vect}(\check u_{F_y}),E_{(0,0,1)}=\mathrm{Vect}(\check u_{F_z}).
 \end{eqnarray*}
 But then $\dim(E)\geq 3$, which is absurd as $\dim(E)=\rk(N_Y)=2$. Hence $\E$ doesn't satisfy Klyachko's criterion and $\cE$ is not locally free.
\end{example}

\subsection{Projectivization of toric vector bundles}
\label{sec:proj bundles}
 
As a motivating example, we consider the variety given as the projectivization of a torus invariant vector bundle. We refer to \cite[Appendix A]{La} and \cite[Section 7.3]{CLS} for the construction and classical results. Let $Y$ be a projective toric variety with fan $\Sigma_Y$. For $i=1,\ldots,r$, let $D_i=\sum_{\rho\in \Sigma_Y(1)} a_{i\rho}D_\rho$ be invariant divisors on $Y$ and let $V_\cF$ be the vector bundle associated to the locally free sheaf
\begin{eqnarray*}
\cF=\cO_Y\oplus\cO_Y(D_1)\oplus\cdots\oplus\cO_Y(D_r). 
\end{eqnarray*}
Let $X=\mathbb{P}(V_\cF^\vee)$ be the projectivization of the dual  of $V_\cF$, with $\pi:X\to Y$ the projection to the base $Y$. Then $X$ is a projective toric variety, with torus $T_X=T_Y\times (\mathbb{C}^*)^r$. Denote by $N_X, N_Y$ the lattices of $X$ and $Y$, and by $M_X$ and $M_Y$ their duals. Then $N_X=N_Y\times \Z^r$ and we have exact sequences
\begin{eqnarray*}
& 0\longrightarrow  \mathbb{Z}^r\longrightarrow N_Y\times \mathbb{Z}^r\longrightarrow N_Y \longrightarrow 0,&\\
 &0\longrightarrow M_Y\longrightarrow M_Y\times(\mathbb{Z}^r)^*\longrightarrow (\mathbb{Z}^r)^*\longrightarrow 0.&
 \end{eqnarray*}
 
 Let $L_Y$ be an ample line bundle on $Y$ and let $\mathcal{O}_X(1)=\cO_{\P(V_\cF)}(1)$ be the Serre line bundle on $X$. Then, up to replacing $L_Y$ by a sufficiently high power, $L_X=\pi^*L_Y\otimes \mathcal{O}_X(1)$ is an ample line bundle on $X$. Assume now that $L_Y=\cO(D_Y)$
 for an ample divisor on $Y$
 $$D_Y=\sum_{\rho\in\Sigma_Y(1)} b_\rho D_\rho,$$
 with associated lattice polytope
 $$
 P_Y=\{m\in (M_Y)_\mathbb{R}\,:\, \langle m,u_\rho\rangle\geq -b_\rho\: \textrm{ for all } \rho\in\Sigma_Y(1)\}.
 $$
 To determine a polytope $P_X\subseteq (M_Y)_\mathbb{R}\times \mathbb{R}^r$ associated to $(X,L_X)$, we first describe the invariant irreducible divisors of $X$. They are of two types:
 $$
 \mathrm{Div}_{irr}^T(X)=\lbrace \pi^{-1}(D_\rho): \rho\in \Sigma_Y(1)\rbrace \cup \lbrace\{s_i =0\}\textrm{ for } i=0,1,\ldots,r \rbrace
 $$
 where the $\lbrace s_i = 0 \rbrace$ are the relative hyperplane sections associated to the line subbundles of $V_\cE^\vee$. Set $\hat D_\rho = \pi^{-1}(D_\rho)$ for all $\rho\in \Sigma_Y(1)$ and $D_{s_i}:=\lbrace s_i = 0 \rbrace$ for all $i\in \lbrace 0, \ldots , r\rbrace$. Then $L_X=\cO(\hat D)$ with 
 $$
 \hat D = \sum_{\rho\in\Sigma_Y(1)} b_\rho \hat D_\rho + \sum_{i=0}^r D_{s_i}
 $$
 and associated polytope $P_X$ defined by $(m_y,m_x)\in (M_Y)_\R\times (\R^r)^*$ is in $P_X$ if and only if 
 \begin{eqnarray*}
 \langle m_x,e_i\rangle \geq -1,  \text{ for all }i=0,\ldots,r, \textrm{ and }\\
 \langle (m_y,m_x), u_\rho +\sum_{i=1}^r a_{i\rho}e_i\rangle \geq -b_\rho,  \text{ for all }\rho\in\Sigma_Y(1),
 \end{eqnarray*}
 where $\{e_i:i=1,\ldots , r\}$ is the standard basis of $\mathbb{R}^r$ and $e_0=-(e_1+\ldots + e_r)$.
 
 We will perform a GIT quotient of $X$ by the torus $G=N_0\otimes_\Z \C^*$, where $N_0=\lbrace 0_Y \rbrace \times\mathbb{Z}^r$. Consider the linearization $\gamma_y$ for $T_Y$ on $L_Y$ giving the polytope $P_Y$ (recall Section \ref{sec:toricGIT}, equation (\ref{eq:linearisation sections})). Consider the standard linearization $\gamma_x$ of $(\C^*)^r$ on $\cO_{\C\P^r}(1)$ associated to the polytope
 $$
 \lbrace \langle m_x,e_i\rangle \geq -1,  \text{ for all }i=0,\ldots,r \rbrace \subset (\R^r)^*.
 $$
 Combining the two gives the linearization $\gamma= \gamma_y \times \gamma_x$ of $T_X$ on $L_X$ associated to $P_X$. Then, setting $U=N_0^\perp\otimes_\mathbb{Z}\mathbb{R}$, we have $U=(M_Y)_\R \times \lbrace 0 \rbrace$. In particular, the only facets of $P_X$ that intersect $U$ are those given by equations $\langle m_y,u_\rho\rangle +\sum_i a_{i\rho}\langle m_x,e_i\rangle=-b_\rho$ for some $\rho\in \Sigma_Y(1)$ and the polytope $P_X \cap U$ is $P_Y\times\lbrace 0 \rbrace$. That is, the GIT quotient of $(X,L_X)$ by $G$ is $(Y,L_Y)$. We also note that $(G,\gamma)$ is generic.
 
\begin{proposition} 
\label{prop:Minkowski for proj bundles}
The polarized toric variety $(X,L_X)$, with toric subgroup $G=(\mathbb{C}^*)^r\subseteq T_X$ and linearization $\gamma$ satisfies the Minkowski condition of Theorem \ref{theo:Minkowski existence}. 
\end{proposition}  
\begin{proof}
 The line bundle $L_X$ determines a polytope in $(M_X)_\mathbb{R}$ whose normal fan equals the fan of $X$. By Corollary \ref{cor:intersection theoretical Minkowski}, we see that
 \begin{eqnarray*}
 \sum_{F\prec P_X}\deg_{L_X}(D_F)\,u_F=0,
 \end{eqnarray*}
and so 
\begin{eqnarray*}
\sum_{F\prec P_X^{s}}\deg_{L_X}(D_F)\,u_F=-\sum_{F\prec P_X^{us}}\deg_{L_X}(D_F)\,u_F.
\end{eqnarray*}
Note that the unstable divisors are precisely the $D_{s_i}$ for $i\in \lbrace 0,\ldots r \rbrace$.
Moreover, the normal elements $u_F$, for $F\prec P_X^{us}$, are given by the basis elements $e_i\in\mathbb{R}^r$ together with $e_0$. Then, $\pi(e_i)=0\in N_Y=N_X/\mathbb{Z}^r$ and hence 
$$\sum_{F\prec P_X^{s}}\deg_{L_X}(D_F)\,u_F=0  \mod\ N_0\otimes_\mathbb{Z}\mathbb{R}$$
as desired.
\end{proof} 
 The GIT quotient map $p:X^s\to Y$ coincides with the restriction $\pi|_{X^s}$ so it makes sense to compare two methods of pulling back sheaves from $Y$ to $X$. Given a $T$-equivariant reflexive sheaf $\cE$ on $Y$, for each $\bfi\in\mathbb{Z}^{r+1}$, we have invariant reflexive sheaves $\pi^*\cE$ and $\mathfrak{P}_\bfi(\cE)$ on $X$.
\begin{proposition}
\label{prop:comparison pullback and functors}
Taking $\bfi=0$, if $\check\cE$ is a $T$-equivariant reflexive sheaf on $Y$, we have
\begin{eqnarray*}
\pi^*\check\cE=\mathfrak{P}_0(\check\cE). 
\end{eqnarray*}
\end{proposition} 
\begin{proof} Suppose that $\check{\mathcal{E}}=\mathfrak{K}(\check{\mathbb{E}})$ where $ \check{\mathbb{E}}=  \{ (E^{\check{F}}(i))\subset E\,:\, \check{F}\prec P_Y,\,i\in\mathbb{Z}  \}$. Then, $\pi^*\check{\mathcal{E}}$ and $\mathfrak{P}_0(\check{\mathcal{E}})$ are determined up to isomorphism by their respective filtrations. On one hand, $\mathfrak{P}_0(\check{\mathbb{E}})= \{ (E^{F}(i))\subset E\,:\, F\prec P_X,\,i\in\mathbb{Z}  \}$ is given by the filtration
\begin{eqnarray*}
E^F(i)= \left\{
\begin{array}{ccc}
 E^{F\cap U}(i) & \text{ for } & F\prec P^s,\, i\in \mathbb{Z},\\
0 & \text{ for } & F\prec P^{us} ,\,  i<0,\\
E & \text{ for } & F\prec P^{us},\, i\geq 0.
\end{array}
\right.
 \end{eqnarray*}   
 On the other hand, by Lemma \ref{lem:pullb functors}, for all $F\prec P_X^s$, observing that $b_F=1$, we have $(\pi^*\check{\mathbb{E}})^F(i)=E^{F\cap U}(i)$. Let $F_i\prec P_X^{us}$ be the unstable facet corresponding to $D_{s_i}$. The morphism $\pi:X\to Y$ corresponds to a lattice homomorphism $\bar{\pi}:N_X\to N_Y$. From the fact that $\bar{\pi}$ sends the normal generator $u_{F_i}=e_i$ of the ray $\rho_{F_i}$ to $0\in N_Y$, we deduce that the image of the affine chart $\pi(U_{\sigma_{F_i}})$ is $T_Y\subseteq Y$. Hence, 
 \begin{eqnarray*}
 \Gamma(U_{\sigma_{F_i}},\pi^*\check{\cE}) = \Gamma(T_Y,\check{\cE})\otimes_{\mathbb{C}[M_Y]}\mathbb{C}[\rho_{F_i}^\vee\cap M_X] = E\otimes \mathbb{C}[\rho_{F_i}^\vee\cap M_X]
 \end{eqnarray*}
 from which it follows that 
\begin{eqnarray*}
(\pi^*\check{\mathbb{E}})^{F_i}(j)=
\begin{cases}
0 & \text{ for } j<0,\\
E & \text{ for } j\geq 0,
\end{cases} 
 \end{eqnarray*}
 allowing us to conclude that $\pi^*\cE=\mathfrak{P}_0(\cE)$. 
 \end{proof}
 From Propositions \ref{prop:Minkowski for proj bundles} and \ref{prop:comparison pullback and functors}, together with Theorem \ref{theo:Minkowski existence}, we obtain
 \begin{proposition}
 \label{prop:ample class proj bundles}
There exists an ample class $\alpha$ on $Y$ such that an equivariant reflexive sheaf $\mathcal{E}$ on $Y$ is $\mu$-stable with respect to $\alpha$ if and only if $\pi^*\mathcal{E}$ is $\mu$-stable on $X$ with respect to $L_X$.
 \end{proposition}
We conclude with a special case where the ample class $\alpha$ can be computed explicitly, that is for $Y$ of dimension $2$ and $V_\cF$ of rank $3$. 
It seems unlikely that a similar method could be used to compute the ample class $\alpha$ for higher dimensions and rank. Nevertheless, the examples below show that in general, $\alpha$ is different from the polarization induced by the GIT quotient.
\begin{lemma}
 \label{lem:orbisurface alpha}
 With previous notations, assume in addition that $Y$ is an orbifold of dimension $2$ and that $V_\cF$ is of rank $3$. Then the class $\alpha$ from Proposition \ref{prop:ample class proj bundles} is equal to $c_1(L_Y^3\otimes \det(V_\cF))$, up to scale.
\end{lemma}
\begin{proof}
 As in that setting $Y$ and $X$ are orbifolds, we can make use of Poincar\'e duality to compute intersections, and degrees. From the proof of Theorem \ref{theo:Minkowski existence}, we see that $\alpha$ is the ample class that satisfies for all $\rho\in \Sigma_Y(1)$,
 $$
 \deg_\alpha(D_\rho)=\deg_{L_X}(\hat D_\rho).
 $$
 We compute the later using that
 $H^2(X,\Z)$ is the algebra over $H^2(Y,\Z)$ generated by the class $\xi:=c_1(\cO_X(1))$, with relation $\xi^3=\pi^*c_1(V_\cF)\cdot \xi^2 - \pi^*c_2(V_\cF)\cdot \xi + \pi^*c_3(V_\cF)$. Taking into account the dimensions, and the rank of $V_\cF$, a direct computation gives
 $$
 \deg_{L_X}(\hat D_\rho)=c_1(L_X)^3\cdot \pi^*c_1(\cO(D_\rho))=\pi^*(3\,c_1(L_Y)+c_1(V_\cF))\cdot \pi^*c_1(\cO(D_\rho))\cdot \xi^2.
 $$
The result follows from Whitney's formula.
\end{proof}

 \bibliographystyle{plain}	

 \bibliography{ClaTipBib}{}

\end{document}